\documentclass[twoside,
]{amsart}

  \usepackage[utf8]{inputenc}
  \usepackage{amsmath}
  \usepackage{amssymb}
  \usepackage{amsthm}
  \usepackage{bm}
  \usepackage{bbm}
  \usepackage{graphicx}
  \usepackage{color}

\newtheorem{Th}{Theorem}
\newtheorem{Prop}{Proposition}
\newtheorem{Lemma}{Lemma}
\newtheorem{Remark}{Remark}
\newtheorem{Def}{Definition}
\newtheorem{Cor}{Corollary}

\def\scalar(#1,#2){(#1\mid#2)}

\renewcommand{\hat}{\widehat}

\newcommand{\cs}{\mathcal{S}}

\newcommand{\cb}{\mathcal{B}}

\newcommand{\ce}{\mathcal{E}}

\newcommand{\cp}{\mathcal{P}}

\newcommand{\ck}{\mathcal{K}}
\newcommand{\xbm}{(X,\mathcal{B},\mu)}

\newcommand{\ycn}{(Y,\mathcal{C},\nu)}

\newcommand{\ov}{\overline}

\newcommand{\bs}{\mathbb{S}}
\newcommand{\bd}{\mathbb{D}}

\newcommand{\Z}{{\mathbb{Z}}}
\newcommand{\N}{{\mathbb{N}}}
\newcommand{\E}{{\mathbb{E}}}

\newcommand{\vep}{\varepsilon}
\newcommand{\va}{\varphi}

\newcommand{\Pa}{\mathcal{P}}

\newcommand{\tend}[3][]{\xrightarrow[#2\to#3]{#1}}
\newcommand{\egdef}{:=}

\newcommand{\ind}[1]{\mathbbmss{1}_{#1}} 

\newcommand{\spdj}{\perp_{\mbox{\scriptsize sp}}}
\newcommand{\Supp}{\Sigma}
\newcommand{\mob}{\boldsymbol{\mu}}
\newcommand{\bun}{\boldsymbol{1}}
\newcommand{\tcs}{T_{\textrm{\scriptsize cs}}}
\newcommand{\ts}{T_{\textrm{\scriptsize symb}}}

\DeclareMathOperator{\Id}{Id}

\title[Spectral disjointness of powers for rank-one transformations]{On spectral disjointness of powers for rank-one transformations and Möbius orthogonality}
\author{El Houcein El Abdalaoui}
\author{Mariusz Lema\'nczyk}
\author{Thierry de la Rue}
\address{El Houcein El Abdalaoui, Thierry de la Rue:
Laboratoire de Math\'ematiques Rapha\"el Salem,
Normandie Université, Universit\'e de Rouen, CNRS --
Avenue de l'Universit\'e --
76801 Saint \'Etienne du Rouvray, France.}
\email{elhoucein.elabdalaoui@univ-rouen.fr\\Thierry.de-la-Rue@univ-rouen.fr}
\address{Mariusz Lema\'nczyk: Faculty of Mathematics and Computer Science, Nicolaus Copernicus University, 12/18 Chopin street, 87-100 Toru\'{n}, Poland}
\email{mlem@mat.umk.pl}
\thanks{Research supported by Narodowe Centrum Nauki grant DEC-2011/03/B/ST1/00407}
\thanks{The research was done during the visit of the second author at Rouen
University in September 2012.}

\begin{document}

\maketitle

\begin{abstract}
  We study the spectral disjointness of the powers of a rank-one transformation. For a large class of rank-one constructions, including those for which the cutting and stacking parameters are bounded, and other examples such as rigid generalized Chacon's maps and Katok's map,  we prove that different positive powers of the transformation are pairwise spectrally disjoint on the continuous part of the spectrum. Our proof involves the existence, in the weak closure of $\{U_T^k:\:k\in\Z\}$, of ``sufficiently many'' analytic functions of the operator $U_T$.
  
  Then we apply these disjointness results to prove Sarnak's conjecture for the (possibly non-uniquely ergodic) symbolic models associated to these rank-one constructions: All sequences realized in these models are orthogonal to the Möbius function. 
  
  \end{abstract}
  
  \renewcommand{\abstractname}{Résumé}
\begin{abstract}
  Nous étudions la disjonction spectrale des puissances d'une transformation de rang un. Pour une large classe de constructions de rang un, incluant celles dont les paramètres de découpage et empilage sont bornés, ainsi que d'autes exemples commes les transformations de Chacon généralisées et la transformation de Katok, nous prouvons que les puissances positives de la transformation sont deux-à-deux spectralement disjointes sur la partie continue du spectre.   
  Notre preuve s'appuie sur l'existence, dans la fermeture faible de $\{U_T^k:\:k\in\Z\}$, de suffisamment de fonctions analytiques de l'opérateur $U_T$. 
  
  Nous appliquons ensuite ces résultats de disjonction pour prouver la conjecture de Sarnak dans les modèles symboliques associés à ces constructions de rang un (qui peuvent ne pas être uniquement ergodiques) : toutes les suites réalisées dans ces modèles sont orthogonales à la fonction de Möbius.
\end{abstract}

\section{Introduction}
\label{sec:Intro}

\subsection{Rank-one automorphisms}
The class of rank-one automorphisms has been under intensive study in ergodic theory for many decades since the works of Baxter~\cite{B}, Chacon~\cite{C}, Katok-Stepin~\cite{K-S1,K-S2} and Ornstein~\cite{Or}, being both a source of interesting examples and constructions, as well as developing its own methods.

Recall that an automorphism $T$ of a given a standard Borel probability space $\xbm$ is said to be \emph{rank-one} if
there exists a sequence of (partial) partitions $(\mathcal{P}_n)$ of $X$ which are Rokhlin towers for the transformation $T$, that is to say of the form $\{F_n,TF_n,\ldots, T^{h_n-1}F_n\}$, such that for each $A\in\cb$
$$ \min\{ \mu(A\triangle A_n),\ A_n\mbox{ measurable with respect to }\mathcal{P}_n\} \tend{n}{\infty} 0. $$
Moreover, we can always assume that the sequence $(\mathcal{P}_n)$ of Rokhlin towers is \emph{increasing}, which means that for each $n$, $F_{n+1}\subset F_n$ and the levels $T^iF_n$ of $\mathcal{P}_n$ are unions of levels of $\mathcal{P}_{n+1}$. The construction of such a transformation can be realized
by the cutting and stacking method~\cite{Fri}, \cite{Na} (Chapter~7), which visually can be described as a series of refining constructions: At stage $n\geq1$ we have a Rokhlin tower $\mathcal{P}_n=\{F_n,TF_n,\ldots, T^{h_n-1}F_n\}$ whose base $F_n$, visualized as an interval, is then divided into $p_n$ equal subintervals, giving rise to $p_n$ columns of $\mathcal{P}_n$ which are numbered from $0$ to $p_n-1$. For each $i=0,\ldots,p_n-1$, above column $i$, we add $s_{n,i}$ spacers, then stack columns with added spacers one above another to obtain $\mathcal{P}_{n+1}$, where the dynamics is to jump by one level up. 
Assuming that the first tower is reduced to a single level $\mathcal{P}_1=\{F_1\}$, as we can always do without loss of generality, the construction of the rank-one automorphism is then completely described by the sequences of parameters $(p_n)_{n\ge 1}$ and $(s_{n,i})_{n\ge 1, 0\le i\le p_n-1}$. 
There are many other ways to define rank-one transformations~\cite{Fe},~\cite{Ju}, including symbolic models which are described in details in Section~\ref{sec:symbolic} (see also Section~\ref{sec:integral automorphisms}).

Despite the fact that the spectral theory of rank-one automorphisms is rather well developed (see e.g.~\cite{Ab0}, \cite{Ab1}, \cite{Ad}, \cite{Ad-Fr}, \cite{Bo0}, \cite{Ch-Na}, \cite{Fe}, \cite{K-S1}, \cite{Kl-Re}, \cite{Na}, \cite{Pr-Ry}, \cite{Ry}), it still possesses many important open problems\footnote{The most famous open problem being whether there exists a rank-one automorphism with Lebesgue spectrum. If this is so, it would give the positive answer to the long-standing Banach problem of existence of automorphism with simple Lebesgue spectrum.}.
In this paper we will be mostly interested in the problem of spectral disjointness of positive\footnote{Recall that any automorphism $T$ is spectrally isomorphic to its inverse.} different powers $T^n$, $n\geq1$, for a rank-one automorphism.
Recall that two automorphisms $T$ and $S$ are \emph{spectrally disjoint} if the maximal spectral types $\sigma_T$ and $\sigma_S$ of their associated Koopman unitary operators $U_T:\ f\mapsto f\circ T$ and $U_S:\ g\mapsto g\circ S$ are mutually singular. (We only consider the action of these operators on the subspace of $L^2$ orthogonal to constant functions.) We write in this case $T\spdj S$. The spectral disjointness question for different powers of $T$ asks whether or not the images of $\sigma_T$ via the maps $z\mapsto z^n$, $z\mapsto z^m$ are mutually singular whenever $n\neq m$. This question is interesting for itself in the class of rank-one transformations, and even though this disjointness property holds for a generic rank-one transformation~\cite{C-N, DJ-L}, there exist weakly mixing rank-one automorphisms whose different positive powers can be even isomorphic~\cite{Ag} (see also~\cite{Dan} and~\cite{Ry3}). 

Observe that if the automorphism $T$ is not weakly mixing, then different powers of $T$ always share some non-trivial eigenvalues, hence are never spectrally disjoint. In the non-weakly mixing case, the question of spectral disjointness of different positive power should therefore be reformulated, and we rather ask whether the \emph{continuous parts} of their maximal spectral types are mutually singular. 

When the present paper was under final redaction, we received the preprint by V.V.~Ryzhikov \cite{Ry4}, in which he proves that all weakly mixing bounded rank-one constructions have disjoint (in the Furstenberg sense) positive powers.

\subsection{Orthogonality with Möbius function - Sarnak's conjecture}
\label{sec:Sarnak}
Another motivation to study spectral disjointness of different powers in the class of rank-one transformations was recently taken up by Bourgain in~\cite{Bo} and deals with Sarnak's conjecture.
Recall that the Möbius function $\mob:\N\to\Z$ is defined by $\mob(1)=1$, $\mob(n)=0$ for non-square-free positive integers, and $\mob(n)=\pm1$ depending on the parity of the number of prime factors for the remaining positive integers. (We use here the bold version of the symbol $\mob$ to distinguish the Möbius function from the probability measure $\mu$.)
Sarnak's conjecture~\cite{Sa} states that for each homeomorphism $T$ of a compact metric space $X$ with zero topological entropy, any $f\in C(X)$ and any $x\in X$, the sequence $\bigl(f(T^nx)\bigr)_{n\geq1}$ is orthogonal to the Möbius function,\textit{ i.e.}
\begin{equation}
\label{eq:s1}
\frac1N\sum_{n=1}^N f(T^nx)\,\mob(n) \tend{N}{\infty} 0.
\end{equation}
Let $(X,T)$ be a uniquely ergodic topological system, and $\mu$ its invariant probability measure. Following~\cite{Bo-Sa-Zi}, we say that $(X,T,\mu)$ has {\em Möbius orthogonal prime powers} 
if  for  each $f\in C(X)$ with $\int_X f\,d\mu=0$ and each $x\in X$, for each sufficiently large different prime numbers $p,q$
\begin{equation}
\label{eq:s2}
\frac1N\sum_{n=1}^Nf(T^{pn}x)f(T^{qn}x)\tend{N}{\infty}0.
\end{equation}
Compared to the classical notion of disjointness in ergodic theory~\cite{Fu}, and following~\cite{Bo-Sa-Zi}, we have the following chain of implications:
\begin{equation}\label{eq:s3}\begin{array}{l}
\mbox{Spectral disjointness of different primes powers}\;\Rightarrow \\
\mbox{Disjointness of different primes powers}\;\Rightarrow\\
\mbox{Möbius orthogonality of prime powers}\;\Rightarrow\;\mbox{Sarnak's conjecture for }(X,T).\end{array}
\end{equation}

In fact, it is proved in~\cite{Bo-Sa-Zi} that Möbius orthogonality of prime powers implies the validity of a generalized version of~\eqref{eq:s1} in which $\mob(n)$ can be replaced by any bounded multiplicative function of the positive integer $n$ (that is, any function $\boldsymbol{\nu}(n)$ satisfying $\boldsymbol{\nu}(n m)=\boldsymbol{\nu}(n)\,\boldsymbol{\nu}(m)$ whenever $n$ and $m$ are coprime).

By considering consecutively an ergodic rotation on finitely many points, and an irrational rotation on the circle, we see that we cannot reverse the last two of the three implications in~(\ref{eq:s3}).
To see that the first implication cannot be reversed either, consider first $S$ with the MSJ property \cite{Ju-Ru} acting on a space $\ycn$ (and having singular spectrum\footnote{It is an open question whether there exists an MSJ automorphism with Lebesgue spectrum. If it exists this would give us an example for which we have disjointness of different positive powers, while all of them are spectrally isomorphic.}). Then $S^p\perp S^q$ for $1\leq p<q$ by~\cite{Ju-Ru}.
Consider now a 2-point extension $T$ (acting on $\xbm$) of $S$ with the property that the spectrum $L^2\xbm\ominus L^2\ycn$ is Lebesgue; this can be done using~\cite{He-Pa}. Now, by \cite{Ju-Ru}, $T$ is simple, and $T^p$ and $T^q$ remain disjoint for $1\leq p<q$ still (because of weak mixing of all transformations considered here, and by the lifting disjointness property by group extensions~\cite{Fu}). On the other hand, $T^p$ is not spectrally disjoint with $T^q$ because both have a Lebesgue component in their spectrum.

When all different positive (prime) powers of an automorphism $T$ are disjoint, Sarnak's conjecture holds for every uniquely ergodic topological model of $T$. However,  it seems unclear whether, if in one uniquely ergodic topological model of an automorphism $T$ Sarnak's conjecture
holds, then it holds in \emph{all} uniquely ergodic topological models of $T$. 

In~\cite{Bo}, Bourgain considers all rank-one constructions in which the cutting and stacking parameters, $p_n$ and $s_{n,i}$, are uniformly bounded and $s_{n,p_n-1}=0$ (that is, with no spacer on the last column), and deals with the uniquely ergodic topological symbolic models of such transformations. He proves some kind of spectral disjointness for different prime numbers rescalings of generalized Riesz products describing the maximal spectral type of a rank-one transformation. This kind of spectral disjointness turns out to be sufficient to obtain a basic Möbius orthogonality lemma proved in~\cite{Bo-Sa-Zi}. As indicated in~\cite{Bo}, the positive answer to Sarnak's conjecture in the class of all rank-one transformations would give automatically the positive answer to Sarnak's 
conjecture for almost every interval
exchange transformations (see~\cite{Ve} for rank-one property of almost every IET). In~\cite{Bo} the 3-IET case is considered.

\subsection{Outline of the paper}

We first develop in Section~\ref{sec:spectral_results} the purely spectral part of our argument. Spectral disjointness of powers of a unitary operator $U$ with continuous maximal spectral type is derived from the existence, in the weak closure of $\{U^k:\:k\in\Z\}$, of ``sufficiently many'' analytic functions of the operator $U$, with a suitable control of their coefficients (see in particular Corollary~\ref{clb2}).

Then we prove that for a large class of rank-one constructions, we can find appropriate weak limits in the weak closure
of $\{U_T^k:\:k\in\Z\}$ to get spectral disjointness (on the continuous part of the spectrum) of different positive powers. In particular, we consider so-called \emph{bounded} rank-one constructions, which simply means that all parameters $(p_n)$ and $(s_{n,i})$ are uniformly bounded (but here we impose no special restriction on the last column, so that the number of consecutive spacers is not necessarily bounded, as spacers coming from different steps of the construction can accumulate). More generally,  we deal with the following generalization of the class of bounded rank-one transformations: The cutting and stacking parameters $(p_n)$ and $(s_{n,i})$ being given, we say that the rank-one construction is \emph{recurrent} if we can find an increasing subsequence $(n_k)$ and  bi-infinite sequences $(\pi_m)_{m\in\Z}$ of positive integers and $(\eta_m)_{m\in\Z}=(\eta_{m,0},\ldots,\eta_{m,\pi_m-1})_{m\in\Z}$ such that, for all $m\in\Z$ and all $0\le j\le \pi_m-1$,
$$ p_{n_k+m} \tend{k}{\infty} \pi_m, \quad\mbox{and } s_{n_k+m,j} \tend{k}{\infty} \eta_{m,j}. $$
Such a subsequence $(n_k)$ will henceforth be called a \emph{stabilizing subsequence}. The rank-one construction is said to be \emph{bounded-recurrent}  if we can find a stabilizing subsequence for which the limit parameters 
$\eta_{n,j}$ are uniformly bounded (we do not require in this definition the boundedness of the parameters $(\pi_m)$).

To obtain appropriate weak limits in the bounded-recurrent case, we use the integral representation over odometer model introduced in~\cite{Ho-Me-Pa} for rank-one automorphisms, and a general theorem describing weak limits in this setting (Theorem~\ref{dflp}). This, together with a ``non-flatness" condition gives the spectral disjointness on the continuous part of the spectrum for different positive powers of a bounded-recurrent rank-one construction (Theorem~\ref{thm:spectral-disjointness}).

We also prove in Section~\ref{sec:unbounded} that our method for proving the spectral disjointness of the powers can also work in classical rank-one examples where $p_n\tend{n}{\infty}\infty$, which of course prohibits the existence of stabilizing subsequences. We consider here the family of rigid generalized Chacon's maps, and Katok's map.

In Section~\ref{sec:symbolic}, we are concerned with the application of our spectral disjointness results to Sarnak's conjecture in the symbolic model associated to a given rank-one construction. Here two kinds of difficulties appear: First, these symbolic models are not in general uniquely ergodic, and we will need an extra combinatorial argument (Proposition~\ref{prop:structure of names}) to deal with the other possible ergodic invariant measure (the Dirac mass on a fixed point). Theorem~\ref{thm:Sarnak for symbolic WM} shows that, at the cost of an extra hypothesis on the control of the parameters $(s_{n,i})$ (which is automatically fulfilled in the bounded case), we can prove Sarnak's conjecture in spite of the lack of unique ergodicity.  Second, since we only get spectral  disjointness on the continuous part of the spectrum, we have to take into account the possible eigenvalues. We treat the case of a finite cyclic group of eigenvalues, which can arise in bounded rank-
one constructions. We finally get Sarnak's conjecture for the symbolic model of any bounded rank-one contruction (Theorem~\ref{thm:Sarnak for bounded rank one}), provided this symbolic model is well defined (that is to say, provided the automorphism is not isomorphic to an odometer).

\subsection*{Acknowledgments}
The authors would like to thank J.-P. Thouvenot for discussions on the subject of the paper.

\section{Weak limits of unitary operators and disjointness of powers}
\label{sec:spectral_results}
Let $U$ be a unitary operator of a separable Hilbert space $H$. We denote by $\sigma_U$ the maximal spectral type of $U$, and we assume throughout that $\sigma_U$ is continuous.

Assume that for some increasing sequence $(h_n)_{n\geq1}$ of natural numbers we have
$$ U^{h_n}\tend{n}{\infty} \sum_{m=-\infty}^\infty a_m U^m$$
weakly in $H$,  for some complex numbers $a_m$, $m\in\Z$, such that
\begin{equation}\label{lb2}
\sum_{m=-\infty}^\infty|a_m|<+\infty.
\end{equation}
Then
\begin{equation}\label{lb1}
z^{h_n}\tend{n}{\infty} \sum_{m=-\infty}^\infty a_m z^m
\end{equation}
weakly in $L^2(\bs^1,\sigma_U)$,  and the formula
\begin{equation}\label{lb3}
\xi(z) \egdef \sum_{m=-\infty}^\infty a_mz^m
\end{equation}
defines a continuous function $\xi$ on the whole circle $\bs^1$. We can interpret the coefficients $a_m$, $m\in\Z$, as the Fourier coefficients of the function $\xi$.

\begin{Prop}\label{plb1}
If $a_m$, $m\in\Z$, in (\ref{lb1}) decreases to zero  exponentially fast as $|m|\to\infty$, then the RHS function defined $\sigma_U$-a.e. in \eqref{lb1} has a (unique) extension to the function $\xi$ (in~(\ref{lb3})) which is an analytic function on the circle $\bs^1$.
\end{Prop}

Let $p\ge2$ be an integer,  and let $\sigma$ be a positive, finite, Borel measure on $\bs^1$. We have the following observation:
\begin{equation}\label{o1}
\begin{array}{l}\mbox{If $\sigma\equiv\sigma_1+\ldots+\sigma_p$ then there exist $\widetilde{\sigma}_i\ll\sigma_i$, $i=1,\ldots,p$,}\\
\mbox{$\widetilde{\sigma}_i\perp\widetilde{\sigma}_j$ whenever $i\neq j$ and $\sigma\equiv\widetilde{\sigma}_1+\ldots+\widetilde{\sigma}_p$}.\end{array}\end{equation}
We denote by $\sigma^{(p)}$ the image of $\sigma$ via the map $z\mapsto z^p$. Then, we claim that there exists $\nu_p\ll\sigma$ such that
\begin{equation}\label{o2} (\nu_p)^{(p)}\equiv \sigma^{(p)},\end{equation}
and
\begin{equation}\label{o3} \mbox{the map $z\mapsto z^p$ is 1-1 $\nu_p$-a.e.}\end{equation}
To see the claim, first write $\sigma\equiv\sigma_1+\ldots+\sigma_p$ with $\sigma_i:=\sigma|_{[(i-1)/p,i/p)}$, $i=1,\ldots,p$ (by abuse of notation, we identify a subinterval of $[0,1)$ with the corresponding arc in $\bs^1$). Then, we have
$$
\sigma^{(p)}\equiv(\sigma_1)^{(p)}+\ldots+(\sigma_p)^{(p)}.$$
Apply~(\ref{o1}) to the latter decomposition to obtain mutually singular $\widetilde{(\sigma_i)^{(p)}}\ll (\sigma_i)^{(p)}$ for $i=0,\ldots,p-1$ satisfying the assertion of~(\ref{o1}). Then pull back the measures $\widetilde{(\sigma_i)^{(p)}}$ on $[(i-1)/p,i/p)$ via the inverse of $z\mapsto z^p$ which is 1-1 to obtain a measure $\eta_i$ and finally set $\nu_p\egdef\eta_1+\ldots+\eta_p$.

\begin{Lemma}\label{lemma:th}
Let $1\le p<q$ be two integers such that $U^p\not\perp_{sp} U^p$.
Then there exist $U$-invariant non-zero subspaces $H_p\subset H$ and $H_q\subset H$, cyclic for $U^p$ and $U^q$ respectively,  such that $U^p|_{H_p}$ is spectrally isomorphic to $U^q|_{H_q}$.
\end{Lemma}
\begin{proof}
By our non-disjointness assumption, there exists $0\neq \nu\ll (\sigma_U)^{(p)}\wedge(\sigma_U)^{(q)}$. Apply the above claim to find $\nu_p\ll\sigma_U$ satisfying~(\ref{o2}) (with $\sigma_U$ in place of $\sigma$) and~(\ref{o3}). Let $G_p$ be a cyclic space for $U$ of spectral type $\nu_p$. Then this space is also $U^p$-invariant. By~(\ref{o2}), the spectral type of $U^p$ on $G_p$ is equal to $(\sigma_U)^{(p)}$, while by~(\ref{o3}), the spectrum of $U^p$ on $G_p$ is simple.

In this way we have found a $U$-invariant (cyclic) space which is also cyclic for $U^p$. Take $H_p\subset G_p$ to be the $U^p$-invariant (necessarily cyclic) subspace $H_p\subset G_p$ of type $\nu$. Since $U^p$ has simple spectrum on $G_p$, $H_p$ is also $U$-invariant. Do the same with $q$ in place of $p$, and find a $U^q$-cyclic space $H_q$ of spectral type $\nu$ which is also $U$-invariant. Clearly, $U^p|_{H_p}$ is isomorphic to $U^q|_{H_q}$.
\end{proof}

Assume now that for each $k\geq1$
\begin{equation}\label{lb6}
U^{kh_n}\tend{n}{\infty} F_k(U)\egdef\sum_{m=-\infty}^\infty a^{(k)}_mU^m
\end{equation}
weakly in $H$. Set $\vep_k\egdef e^{2\pi i/k}$.
Given $z=e^{2\pi i\theta}\in\bs^1$ with $\theta\in[0,1)$, we also write $r^{(k)}_j(z)\egdef e^{2\pi i((\theta+j)/k)}$, $j=0,\ldots,k-1$.

Note that for $1\le p<q$, $0\le j\le p-1$, and for each $z\in\bs^1$,
\begin{equation}\label{lb6a}
\left(r^{(pq)}_{0}(z)\right)^q=r^{(p)}_j(z)\vep_p^{-j}.\end{equation}
We obtain a symmetric formula for $q$.

\begin{Prop}\label{llb1}
Fix $1\leq p<q$  and assume that $F_p$ and $F_q$ are analytic on $\bs^1$ (that is, the coefficients $a^{(p)}_m,a^{(q)}_m$, $m\in\Z$, decrease exponentially fast). If, for each  $j=0,\ldots,p-1$, $k=0,\ldots,q-1$, the analytic functions
$$ z\mapsto F_p(\vep_p^j z^{q})\quad \mbox{and}\quad z\mapsto F_q(\vep_q^k z^{p})$$
are different,  then $U^p$ and $U^q$ are spectrally disjoint. If $p$ and $q$ are coprime then we only need to check that $F_p(z^q)\neq F_q(z^p)$.
\end{Prop}
\begin{proof}
Suppose that $U^p\not\perp_{sp} U^q$. Then, by Lemma~\ref{lemma:th}, we can find a $U^p$-invariant cyclic subspace $H_p\subset H$ and a $U^q$-invariant cyclic subspace $H_q\subset H$ such that $(H_p,U^p)$ and $(H_q,U^q)$ are isomorphic, and both $H_p$ and $H_q$ are $U$-invariant. Let $\sigma'$ be the common maximal spectral type of $U^p|_{H_p}$ and $U^q|_{H_q}$.  By passing to the spectral model, the action of $U^p$ on $H_p$ and the action of $U^q$ on $H_q$ are both represented in $L^2(\bs^1,\sigma')$ by multiplication by $z$. The action of $U$ on $H_p$ (respectively $H_q$) is represented in $L^2(\bs^1,\sigma')$ by multiplication by some function $\va_p$ (respectively $\va_q$) of modulus one. Moreover, in view of~(\ref{lb6}) in $H_p$ (with $k=p$), $z^{h_n}$ converges weakly in $L^2(\bs^1,\sigma')$ to $F_p(\va_p(z))$, while by~(\ref{lb6}) in $H_q$ (with $k=q$), it also converges weakly in $L^2(\bs^1,\sigma')$ to $F_q(\va_q(z))$. Therefore, we have
\begin{equation}\label{lb10}
F_p(\va_p(z))=F_q(\va_q(z))\;\;\mbox{for $\sigma'$-a.e}\;\;z\in\bs^1.\end{equation}
Moreover
$$ \va_p(z)^p=z=\va_q(z)^q\quad(\sigma'\mbox{-a.e.}).$$
It follows that
$$\va_p(z)\in\{r^{(p)}_0(z), \ldots, r^{(p)}_{p-1}(z)\}\quad(\sigma'\mbox{-a.e.})$$
and
$$\va_q(z)\in\{r^{(q)}_0(z), \ldots, r^{(q)}_{q-1}(z)\}\quad(\sigma'\mbox{-a.e.})$$
 Thus, there exist $j\in\{0,\ldots,p-1\}$, $k\in\{0,\ldots,q-1\}$ and $A\subset\bs^1$, $\sigma'(A)>0$, such that
\begin{equation}\label{lb12}
\va_p(z)=r^{(p)}_j(z),\quad \va_q(z)=r^{(q)}_k(z)\quad\mbox{for each}\;\;z\in A.
\end{equation}
Set $\Theta:A\to\bs^1$,
$\Theta(z)\egdef r^{(pq)}_0(z)$.
In view of~(\ref{lb6a}) and~(\ref{lb12})
$$
\Theta(z)^q\vep^{j}_p=\va_p(z),\quad\Theta(z)^p\vep_{q}^k=\va_q(z),\quad z\in A.$$
Now, defining $\nu$ as the image of the measure $\sigma'|_{A}$ via the map
$A\ni z\mapsto \Theta(z)\in\bs^1$,  we obtain by~\eqref{lb10} that $F_p(\vep^j_pw^q)=F_q(\vep_q^kw^p)$ for $\nu$-a.e.\ $w\in\bs^1$. Since $F_p,F_q$ are analytic, and $\nu$ is a non-zero continuous measure, the two analytic functions $$ z\mapsto F_p(\vep_p^j z^{q})\quad \mbox{and}\quad z\mapsto F_q(\vep_q^k z^{p})$$
coincide. This proves the first part of the lemma.

If additionally, $p$ and $q$ are coprime, there is a unique $m\in\{0,\ldots,pq-1\}$ such that for the map $\Theta:A\to\bs^1$ defined by
$\Theta(z)\egdef r^{(pq)}_m(z)$
we have
$$
\Theta(z)^q=r^{(p)}_j(z),\quad\Theta(z)^p=r^{(q)}_k(z),\quad z\in A.$$
(Indeed, $m$ satisfies $m=j\mod p$ and $m=k\mod q$, so the assertion follows by the Chinese remainder theorem.) We conclude, as in the general case, that $F_p(w^q)=F_q(w^p)$.
\end{proof}

Set $\Supp(F_k)\egdef \{n\in\Z:\: a^{(k)}_n\neq0\}$ and notice that
$$
\Supp\left(F_p(\vep^j_p(\,\cdot\,)^q)\right)=q\Supp(F_p).$$

\begin{Cor}\label{clb2}
Assuming that $F_p$ and $F_q$ are analytic, if $q\Supp(F_p)\neq p\Supp(F_q)$, then $U^p\perp_{sp} U^q$. In particular, if $1\in \bigl(\Supp(F_p)-\Supp(F_p)\bigr)\cap\bigl(\Supp(F_q)-\Supp(F_q)\bigr)$, then $U^p\perp_{sp} U^q$.
\end{Cor}

\begin{Remark}\label{remS}\em In order to prove  Sarnak's conjecture for some automorphisms, we are interested in spectral disjointness $U^p\perp_{sp} U^q$ only for prime powers. In particular, $p$ and $q$ are coprime. Under the latter assumption, the equality $q\Supp(F_p)=p\Supp(F_q)$ (see Proposition~\ref{llb1}) yields $\Supp(F_p)\subset p\Z$ and $\Supp(F_q)\subset q\Z$. Moreover, $F_p\bigl((\,\cdot\,)^q\bigr)=F_q\bigl((\,\cdot\,)^p\bigr)$ implies
$$
a^{(p)}_{ps}=a^{(q)}_{qs}\;\;\mbox{for each}\;\;s\in\Z.$$
It follows that, for $p\neq q$ prime numbers and assuming the analyticity of $F_p$ and $F_q$, $U^p\not\perp_{sp} U^q$ implies the existence of a sequence $(b_s)_{s\in\Z}$ such that $F_p$ and $F_q$ are of the form
\begin{equation}\label{lbsarnak}
F_p(z) = \sum_{s=-\infty}^\infty b_s z^{ps},\quad F_q(z) = \sum_{s=-\infty}^\infty b_s z^{qs}.\end{equation}
\end{Remark}

\section{A weak convergence theorem for integral automorphisms}

Let $S$ be an ergodic automorphism of a standard Borel probability space $\ycn$. Let $f:Y\to\Z_{+}^\ast=\{1,2,\ldots\}$ be integrable. Set
$$
Y^f\egdef\{(y,i):\:y\in Y, 0\leq i < f(y)\},$$
with the natural product structure, the Borel $\sigma$-algebra $\cb^f$ and the probability measure $\nu^f$ defined by
$$ \nu^f(A)\egdef \dfrac{1}{\int_Y f\,d\nu} \sum_{i\ge 0} \nu\bigl(\{y\in Y:\ (y,i)\in A\}\bigr). $$
We consider on $Y^f$ the vertical action $S^f:\ (y,i)\mapsto(y,i+1)$, where we identify $(y,f(y))$ with $(Sy,0)$.
Then $S^f$ is an ergodic automorphism of $(Y^f, \cb^f, \nu^f)$ called \emph{integral automorphism over $S$}.

Let $g:\ Y\to \Z$. For each $q\ge1$, we set
$$ g^{(q,S)}\egdef g+g\circ S+\cdots+ g\circ S^{q-1}. $$
(If it is clear which underlying transformation $S$ we refer to, we will only write $g^{(q)}$ instead of $g^{(q,S)}$.) Let $\cp(\Z)$ be the space of all probability distributions on $\Z$. We denote by $g_\ast\in \cp(\Z)$ the probability distribution of $g$ (that is, the image of $\nu$ under the map $g$).

Now we also assume that $(Y,d)$ is a compact metric space, and that $f$ is square integrable. Suppose that $S$ is uniformly rigid along an increasing sequence $(q_n)$, \textit{i.e.}\ $d(S^{q_n}y,y)\to0$ uniformly, when $n\to\infty$.
Then by Theorem~6 in~\cite{Fr-Le} (with $C_n=Y$) combined with Lemma~3, Lemma~31 and the proof of Proposition~32 in~\cite{De-Fr-Le-Pa} we obtain the following result.

\begin{Th}\label{dflp} Under the above assumptions on $S$ and $f$, assume that there exist $M>0$, $P\in\cp(\Z)$, and a sequence $(h_n)\subset\Z$ such that $\|f^{(q_n)}-h_n\|_{L^2}\leq M$, $n\geq1$, and
\begin{equation}\label{dflp1}
\left(f^{(q_n)}-h_n\right)_\ast\tend{n}{\infty} P\;\;\mbox{weakly in}\;\;\cp(\Z).\end{equation}
Then
\begin{equation}\label{dflp2}
(S^f)^{-h_n}\tend{n}{\infty}\sum_{k\in\Z}P(\{k\})(S^f)^k\end{equation}
weakly in the space of Markov operators.
\end{Th}

\begin{Remark}\label{dflprem}
\em
Note that when $\|(f^{(q_n)}(\cdot)-h_n\|_{L^2}\leq M$, we also have
$$
\|(f^{(jq_n)}-jh_n\|_{L^2}\leq jM\;\;\mbox{ for all}\;\;j\geq1$$
(indeed, we have $f^{(jq_n)}=f^{(q_n)}+f^{(q_n)}\circ S^{q_n}+\ldots+f^{(q_n)}\circ S^{(j-1)q_n}$). Therefore, by passing to a subsequence if necessary, the sequence
$\left( (f^{(jq_n)}-jh_n)_\ast\right)_{n\ge1}$ converges weakly to some probability measure $P_j\in\mathcal{P}(\Z)$ and the assertion of Theorem~\ref{dflp} holds with $h_n$ replaced by $jh_n$ and $P$ replaced by $P_j$. In general, there is no connection between $P$ and $P_j$.
\end{Remark}

\section{Rank-one transformations as integral automorphisms and weak convergences} 
\label{sec:integral automorphisms}
For the representation of a rank-one transformation as an integral automorphism over an odometer, we follow~\cite{Na} (see also \textit{e.g.} \cite{Ho-Me-Pa}, \cite{Pr-Ry}).

Assume that $p_j\geq2$ are integers for $j\geq1$. Set $q_n\egdef p_1 p_2\cdots  p_n$, $n\geq1$. Let $(Y,S)$ be the corresponding $(q_n)$-odometer. That is, we define $Y$ as the compact metric space $ Y\egdef\Pi_{n=1}^\infty\{0,\ldots,p_n-1\}$. We endow it with the topological group structure by adding the coordinates modulo~$p_n$ and transferring the carry to the right. We also consider on $Y$ the Haar probability measure $\nu$, under which all coordinates are uniformly distributed (in their respective integer intervals) and independent. This measure is preserved by the odometer transformation $S$, defined by $S(y)\egdef y+\hat{1}$ with $\hat{1}\egdef (1,0,0,\ldots)$.  
It is a rank-one transformation itself, where no spacers are added. 
We have a refining sequence of towers $\mathcal{D}_n=\{D^{(n)}_0,\ldots,D^{(n)}_{q_n-1}\}$, $n\geq 1$, fulfilling the whole space and tending to the partition into points, defined by $D^{(n)}_0\egdef\{y\in Y:\: y_1=\ldots=y_n=0\}$ and $D^{(n)}_{i}\egdef S^{i}D^{(n)}_0$  for $i=0,\ldots,q_n-1$\footnote{\label{stopka1}Yet, another partition of $Y$ at stage $n\geq1$ can be considered: It is given by $n$-columns $C^{(n)}_i:=\bigcup_{r=0}^{q_{n}-1}S^rD^{(n)}_{iq_n}$, 
 $i=0,1,\ldots,p_{n+1}-1$; notice that $y\in C^{(n)}_i$ if and only if $y_{n+1}=i$.}.
Note that $\mathcal{D}_n$ is the partition generated by the $n$ first coordinates of $y\in Y$, and that
$$
S^{q_n}D^{(n)}_{i}=D^{(n)}_{i}\quad (0\le i\le q_n-1).$$
Moreover,
\begin{equation}
  \label{eq:top}
D^{(n)}_{q_n-1}=D^{(n+1)}_{q_n-1}\cup D^{(n+1)}_{2q_n-1}\cup\ldots\cup D^{(n+1)}_{p_{n+1}q_n-1}
\end{equation}
and
\begin{align}
\label{ooo}
S^{q_n}D^{(n+1)}_{jq_n-1}&=D^{(n+1)}_{(j+1)q_n-1},\quad j=1,\ldots,p_{n+1}-1,\\ S^{q_n}D^{(n+1)}_{p_{n+1}q_n-1}&=D^{(n+1)}_{q_n-1}.\notag
\end{align}

Assume that $f:Y\to\Z^+_\ast$ is an $L^1$-function. We call it of {\em Morse type} if it can be represented as
\begin{equation}\label{moa}f=1+\sum_{n=1}^\infty  s_n\end{equation}
where
$s_n:Y\to \Z^+=\{0,1,2,\ldots\}$ is $\mathcal{D}_n$-measurable, $n\geq1$, and
\begin{equation}\label{mo2}
\mbox{supp\,$s_n\subset D^{(n-1)}_{q_{n-1}-1}$ for $n\geq2$}.
\end{equation}
By~\eqref{eq:top}, the function $s_n$ is completely defined by the values it takes on $D_{q_{n-1}-1}^{(n)}$,  $D_{2q_{n-1}-1}^{(n)},\ldots,$  $D_{p_nq_{n-1}-1}^{(n)}$, which we respectively denote by $s_{n,0}$, $s_{n,1},\ldots$, $s_{n,{p_n-1}}$.

Set $h_1\egdef 1$, and for $n\geq1$ set
\begin{equation}\label{mo3}
h_{n+1}\egdef p_{n}h_n+\sum_{j=0}^{p_n-1}s_{n,j}.\end{equation}

When $f$ is a Morse type function, it is not hard to see that the integral automorphism $S^f$ is rank-one. Indeed, the sequence of towers for $S^f$ is obtained consecutively by taking
the towers with base $D^{(n)}_0\times\{0\}$. These towers have heights given by $h_n$ above.

As a matter of fact, each rank-one transformation is of the form $S^f$ for some odometer $S$ and $f$ of Morse type~\cite{Ho-Me-Pa},~\cite{Na}: $p_n$ are given by the number of columns in the construction, while the functions $s_n$ are given by the number of spacers over a column, the sequence $(h_n)$ of heights of the towers is obtained from the recursive formula~(\ref{mo3}) relating the heights of the towers with the parameters giving the number of divisions and the spacers.

\begin{Lemma}\label{mol1}
For each $n\geq1$ the function $s_n^{(q_n)}(\cdot)$ is constant on $Y$ and
\begin{equation}\label{mo4}
s_n^{(q_n)}(y)=\sum_{j=0}^{p_n-1}s_{n,j}.
\end{equation}
Moreover,
\begin{equation}\label{mo5}
s_n^{(jq_n)}(y)=js_n^{(q_n)}\;\;\mbox{for each}\;\;j\geq1.\end{equation}
We also have
\begin{equation}\label{mo6}
h_{n+1}=p_{n}h_n+s_n^{(q_n)}.\end{equation}\end{Lemma}
\begin{proof}
The orbit $\{y,Sy,\ldots,S^{q_n-1}y\}$ of each point $y\in Y$ meets every set $D^{(n)}_{iq_{n-1}-1}$ exactly once for each $i=1,\ldots,p_n$. Hence~(\ref{mo4}) follows. Also~(\ref{mo5}) follows as $s_n^{(jq_n)}(y)=\sum_{m=0}^{j-1}s_n^{(q_n)}(S^{mq_n}y)$. Finally,~(\ref{mo6}) follows from~(\ref{mo4}) and~(\ref{mo3}).
\end{proof}

\begin{Lemma}\label{mol2} We have
$$(1+s_1+s_2+\ldots+s_n)^{(q_n)}(y)=h_{n+1}$$ for each $y\in Y$.
\end{Lemma}
\begin{proof} Indeed, by Lemma~\ref{mol1} and the fact that $q_1=p_1$, $$(1+s_1)^{(q_1)}=q_1+s_1^{(q_1)}=p_1+\sum_{j=0}^{p_1-1}s_{1,j}=h_2.$$
Assume now the lemma has been proved for some $n\ge 1$. Then, by Lemma~\ref{mol1}, it follows that
\begin{align*}
&(1+s_1+s_2+\ldots+s_n+s_{n+1})^{(q_{n+1})}(y)\\
&=(1+s_1+s_2+\ldots+s_n)^{(p_{n+1}q_n)}(y)+s_{n+1}^{(q_{n+1})}(y)\\
&=\sum_{m=0}^{p_{n+1}-1}(1+s_1+s_2+\ldots+s_n)^{(q_n)}(S^{mq_{n}}y)+
\sum_{j=0}^{p_{n+1}-1}s_{n+1,j}\\
&=p_{n+1}h_{n+1}+\sum_{j=0}^{p_{n+1}-1}s_{n+1,j}=h_{n+2}.
\end{align*}
\end{proof}

Set $f_{n+1}\egdef \sum_{m\ge n+1}s_{m}$. Then supp\,$f_{n+1}\subset D^{(n)}_{q_n-1}$. Define also $g_{n+1}(y)\egdef f_{n+1}(S^iy)$, where $0\leq i<q_n$ is unique to satisfy $S^iy\in D^{(n)}_{q_n-1}$, that is, we spread the values of $f_{n+1}$ along the columns $C^{(n)}_i$, see footnote~\ref{stopka1}: for $z\in D^{(n)}_{iq_n}$ we set
$$
g_{n+1}(z)=g_{n+1}(Sz)=\ldots=g_{n+1}(S^{q_n-1}z)=f_{n+1}(S^{q_n-1}z).$$

\begin{Lemma}\label{mol3}For each $j\geq1$ we have
$$
f_{n+1}^{(jq_n)}=g_{n+1}^{(j,S^{q_n})}.$$
\end{Lemma}
\begin{proof}Since by~(\ref{mo2}),
 supp\,$f_{n+1}\subset D^{(n)}_{q_n-1}$, for each $y\in Y$ the orbit $\{y,\ldots, S^{q_n-1}y\}$ meets the support exactly once, so the result holds for $j=1$.

 In the general case, let $0\leq i<q_n$ be such that $S^iy\in D^{(n)}_{q_n-1}$. Then the only points in the orbit $\{y,\ldots, S^{jq_n-1}y\}$ that meet $D^{(n)}_{q_n-1}$ are of the form $S^{i+kq_n}y$ and, by~(\ref{ooo}), we have
 $$
 \sum_{k=0}^{j-1}f_{n+1}(S^{i+kq_n}y)=\sum_{r=0}^{j-1}g_{n+1}(S^{rq_n}y)$$
 which completes the proof.
\end{proof}

Using the above and the proof~\footnote{One can apply Theorem~\ref{dflp} directly to all rank-one transformations for which in their constructions  there are no spacers put over the the last column, and the number of spacers is bounded; indeed, in this case the functions $g_n$ are commonly bounded.}  of Theorem~\ref{dflp} we obtain the following result.
\begin{Prop}\label{mop1} For each $j\geq1$ and $n\geq1$ we have
$$
f^{(jq_n)}-jh_{n+1}=g^{(j,S^{q_n})}_{n+1}.$$
If, moreover,
$\left(f^{(jq_{n_k})}-jh_{n_k+1}\right)_\ast\tend{k}{\infty} P_j$ weakly in $\mathcal{P}(\Z)$ then $P_j(\{\ldots,-2,-1\})=0$,
$$
(S^f)^{-jh_{n_k}}\tend{k}{\infty} \sum_{r=0}^\infty P_j(\{r\})(S^f)^r, $$
and the function $\sum_{r=0}^\infty P_j(\{r\})z^r$ is analytic in $\bd$. If the sequence $P_j(\{r\})$, $r\geq0$ decreases exponentially fast, then $\sum_{r=0}^\infty P_j(\{r\})z^r$ is analytic in $\ov{\bd}$.\end{Prop}
\begin{proof} For the first part of the proposition, notice that by Lemmas~\ref{mol2} and~\ref{mol3}, we have
$$
f^{(jq_n)}=(1+s_1+\ldots+s_n)^{(jq_n)}+f_{n+1}^{(jq_n)}=jh_{n+1}+g^{(j,S^{q_n})}_{n+1}.$$
Moreover, since $g_{n+1}$ takes only non-negative values, $P_j(\{\ldots,-2,-1\})=0$.

For the second part, we consider only $j=1$ (the proof is the same for all $j\geq1$). If the sequence $\|f^{(q_{n_k})}-h_{n_k+1}\|_{L^2}=\|g_{n_k+1}\|_{L^2}$, $k\geq1$, is bounded in $L^2\ycn$, the result follows directly from Theorem~\ref{dflp}. If not, we take $\vep>0$ and find $M>0$ so that $P_1([0,M]\cap \Z)>1-\vep$, that is
$$
\nu(C_k)>1-\vep,\;\;\mbox{where}\;\;C_k=\{y\in Y:\:g_{n_k+1}\leq M\}$$
for all $k$ large enough, say $k\geq K$. Notice that the function $g_{n_k+1}$ is ``almost'' $S$-invariant, because for each $k\geq1$, for all  of the points $y$ except those belonging to the top level $D^{(n_k)}_{q_{n_k}-1}$, $y$ and $Sy$ are in the same column in $C^{(n_k)}_r$ (see footnote~\ref{stopka1}), whence, by its definition,  $g_{n_k+1}(y)=g_{n_{k}+1}(Sy)$. It follows that $\nu(C_k\triangle SC_k)\to0$ when $k\to\infty$. Since now the functions $(f^{(q_{n_k})}-h_{n_k+1})|_{C_k}$, $k\geq K$, are commonly bounded, it follows that (by passing to a further subsequence if necessary) by Theorem~6 in~\cite{Fr-Le} (and using the arguments from~\cite{De-Fr-Le-Pa})
$$
(S^f)^{-h_{n_k+1}}\to (1-\vep)\sum_{r=0}^\infty P'_{\vep}(\{r\})(S^f)^r+\vep J,$$
where $P'_{\vep}\in\mathcal{P}(\Z)$ and $J$ is a Markov operator. Moreover, $P'_{\vep}\to P$, when $\vep\to0$. By passing to a further subsequence if necessary, we obtain the result.
\end{proof}

\section{Recurrent rank-one constructions and spectral disjointness of powers}

We now assume the existence of a stabilizing subsequence for our parameters $(p_n)$ and $(s_{n,j})$. Recall from Section~\ref{sec:Intro} that it means we can find bi-infinite sequences $(\pi_m)_{m\in\Z}$ and $(\eta_m)_{m\in\Z}=(\eta_{m,0},\ldots,\eta_{m,\pi_m-1})$, and a subsequence $(n_k)_{k\ge1}$ such that, for all $m\in\Z$ and all $1\le j\le \pi_m$,
$$ p_{n_k+m} \tend{k}{\infty} \pi_m, \quad s_{n_k+m,j} \tend{k}{\infty} \eta_{m,j}. $$
We also interpret $\eta_m$ as a function mapping $j\in\{0,\ldots,\pi_m-1\}$ to $\eta_{m,j}$, and $s_n$ as a function mapping $j\in\{0,\ldots,p_n-1\}$ to $s_{n,j}$.

Let $j\ge 1$ be a fixed integer. We are interested in the limit distribution, as $k\to\infty$, of
$$ f^{(jq_{n_k})}-jh_{n_k+1} = 
g_{n_k+1}^{(j,S^{q_{n_k}})} $$
(cf.\ Proposition~\ref{mop1}).

Observe (see footnote~\ref{stopka1}) that for each $y=(y_1,y_2,\ldots)\in Y$,
\begin{equation}
\label{eq:formula_for_g}
  g_{n_k+1}(y) = \sum_{m=1}^t s_{n_k+m}(y_{n_k+m}),
\end{equation}
where $t=t(n_k,y)$ is the smallest positive integer such that $y_{n_k+t}<p_{n_k+t}-1$. Note also that, for each $r\ge 1$ the probability that $t(n_k,y)>r$ is uniformly bounded by~$2^{-r}$.

Let $\ov Y\egdef\prod_{m\ge 1}\{0,\ldots,\pi_m-1\}$, and define $\gamma:\ \ov Y\to\Z_+$ by
\begin{equation}
 \label{eq:def_gamma}
\gamma (\ov y) \egdef \sum_{m=1}^{\ov t} \eta_{m}(\ov y_m),
\end{equation}
where $\ov t=\ov t(\ov y)$ is the smallest positive integer such that $\ov y_{\ov t}<\pi_{\ov t}-1$. We can also view $\eta_m$ as a function defined on $\ov Y$, supported on the set of $\ov y$ with $\ov y_1=\pi_1-1,\ldots,\ov y_{m-1}=\pi_{m-1}-1$, defined by 
$$ \eta_m(\pi_1-1,\ldots,\pi_{m-1}-1,\ov y_m,\ov y_{m+1},\ldots)\egdef\eta_{m}(\ov y_m), $$
so that (\ref{eq:def_gamma}) becomes
\begin{equation}
 \label{eq:def_gamma_bis}
\gamma (\ov y) = \sum_{m\ge1} \eta_{m}(\ov y).
\end{equation}

We also introduce the odometer transformation $\ov S:\ \ov Y\to \ov Y$. For all $j\ge 1$, we denote by $P_j$ the distribution of 
$\gamma^{(j)}=\gamma+\gamma\circ \ov S+\cdots + \gamma \circ \ov S^{j-1}$, when $\ov Y$ is endowed with its Haar measure $\ov \nu$.

\begin{Lemma}
  \label{lemma:convergence_to_P_j}
  For each $j\ge 1$, the distribution of 
  $$ f^{(jq_{n_k})}-jh_{n_k+1} = g_{n_k+1}^{(j,S^{q_{n_k}})} $$
  converges weakly in $\cp (\Z)$ to $P_j$ as $k\to\infty$.
\end{Lemma}

\begin{proof}
  Fix a large integer $r$. Then choose $k$ large enough, so that for each $0\le m\le r+1$, $p_{n_k+m}$ and $s_{n_k+m}$ coincide respectively with  $\pi_m$ and $\eta_m$. Comparing \eqref{eq:formula_for_g} and  \eqref{eq:def_gamma}, we see that the distribution of $g_{n_k+1}$ on the set 
  $$ B(n_k,r)\egdef\left\{y\in Y:\ t(n_k,y)\le r \right\}$$ 
  coincide with  the distribution of $\gamma$ on the set 
  $$\ov B(r)\egdef\left\{ \ov y\in\ov Y:\ \ov t(\ov y)\le r \right\}.$$ 
  Moreover, these two sets have the same measure which is at least $1-2^{-r}$. This proves the lemma for $j=1$. For the general case,  note that $S^{q_{n_k}}$ acts on $y=(y_1,y_2,\ldots)\in Y$ by adding $1$ on the coordinate $y_{n_k+1}$ and transferring the carry to the right.  Then the distribution of $g_{n_k+1}^{(j,S^{q_{n_k}})}$ on $\bigcap_{\ell=0}^{j-1}S^{-\ell q_{n_k}}B(n_k,r)$ coincides with the distribution of $\gamma$ on $\bigcap_{\ell=0}^{j-1}\ov S^{-\ell}\ov B(r)$. Since the common measure of these two sets is at least $1-j2^{-r}$, which can be made arbitrarily close to 1 by fixing $r$ large enough, this proves the lemma.
\end{proof}

For each $m\in\Z$, we introduce the following sets:
$$ \cs_m \egdef \eta_m\bigl(\{0,\ldots,\pi_m-2\}\bigr) \cup \Bigl( \eta_{m,\pi_m-1}+\eta_{m+1}\bigl(\{0,\ldots,\pi_{m+1}-2\}\bigr)\Bigr)$$
and
$$ \ce_m \egdef \cs_m-\cs_m. $$

The set $\cs_m$ can be interpreted as the set of values of $\sum_{r\ge m}\eta_r$ over the set of all $\ov y\in \ov Y$ satisfying $\ov y_1=\pi_1-1,\ldots,\ov y_{m-1}=\pi_{m-1}-1,\ov y_{m+1}<\pi_{m+1}-1$. As the following lemmas show, the set $\ce_m$ of differences between elements of $\cs_m$ is very useful to describe some properties of the distribution $P_j$.

\begin{Lemma}
\label{lemma:distribution_of_Pj}
  Let $j\ge 1$ and $d\ge 1$ be fixed integers. Assume that there exists $m$ with $2^{m-1}>j$ such that $d\in \ce_m$. Then there exist $\alpha,\beta\in\Z$ with $\alpha-\beta=d$, $P_j(\alpha)>0$ and $P_j(\beta)>0$.
\end{Lemma}

\begin{proof}
By the definition of $\cs_m$, there exist $\ov y=(\ov y_1,\ov y_2,\ldots)$ and $\ov z=(\ov z_1,\ov z_2,\ldots)$ in $\ov Y$ with $\ov y_r = \ov z_r = \pi_r-1$ ($1\le r \le m-1$), $\ov y_{m+1} < \pi_{m+1}-1$ and  $\ov z_{m+1} < \pi_{m+1}-1$, such that
$$ \sum_{r\ge m}\eta_r(\ov z) - \sum_{r\ge m}\eta_r(\ov y)=d. $$
Since $j < 2^{m-1} \le \pi_1\cdots\pi_{m-1}$, for any $\ell\in\{1,\ldots,j-1\}$ we have
$$\sum_{r\ge m} \eta_r (\ov S^\ell \ov y) =  \sum_{r\ge m} \eta_r (\ov S^\ell \ov z)  =  0. $$
Moreover, observe that the first $(m-1)$ coordinates of $\ov S^{\ell}\ov y$ coincide with the first $(m-1)$ coordinates of $\ov S^{\ell}\ov z$. Hence, for any $0\le \ell\le j-1$,
$$ \sum_{1\le r\le m-1} \eta_r (\ov S^\ell \ov y) =  \sum_{1\le r\le m-1} \eta_r (\ov S^\ell \ov z).$$
It follows that
\begin{equation}
\label{eq:diff_d}
 \sum_{0\le\ell\le j-1} \left( \gamma(\ov S^\ell \ov z) - \gamma(\ov S^\ell \ov y)\right) = d.
\end{equation}
Finally, note that changing the coordinates $\ov y_r$ and $\ov z_r$ for $r>m+1$ does not affect the above equality, hence there is a positive measure set of $\ov y\in \ov Y$ and a positive measure set of $\ov z \in \ov Y$ for which (\ref{eq:diff_d}) holds.
\end{proof}

\begin{Lemma}
\label{lemma:multiple_of_d}
  Let $d\ge1$ be such that, for any $m\in\Z$, $\ce_m\subset d\Z$. Let $j\ge 1$ and $\alpha,\beta\in\Z$ with $P_j(\alpha)>0$ and $P_j(\beta)>0$. Then $\alpha-\beta$ is a multiple of $d$.
\end{Lemma}
\begin{proof}
It is enough to prove that $\gamma(\ov z)-\gamma(\ov y)\in d\Z$ for all $\ov y, \ov z\in \ov Y$.
Let $t=t(\ov y)$ be the smallest positive integer such that $\ov y_t < \pi_t-1$ and $t'=t(\ov z)$ be the smallest positive integer such that $\ov z_{t'} < \pi_{t'}-1$. Suppose first that $t=t'$. Then by (\ref{eq:def_gamma}), we get
$$ \gamma(\ov z)-\gamma(\ov y)= \eta_t(\ov z_t) - \eta_t(\ov y_t) \in \ce_t\subset d\Z. $$
Suppose now that $t<t'$. Then, since the coordinates of $\ov y$ and $\ov z$ coincide up to $t-1$, (\ref{eq:def_gamma}) gives
\begin{eqnarray*}
 \gamma(\ov z)-\gamma(\ov y) &=& \left( \sum_{m=t}^{t'-1} \eta_m (\pi_m-1) + \eta_{t'}(\ov z_{t'})\right) - \eta_t(\ov y_t) \\
&=& \eta_{t'}(\ov z_{t'}) + \eta_{t'-1}(\pi_{t'-1}-1) - \eta_{t'-1}(0) \\
&& + \sum_{m=2}^{t'-2} \eta_{m+1}(0) + \eta_m(\pi_m-1) - \eta_m(0) \\
&& + \eta_t(0) - \eta_t(\ov y_t).
\end{eqnarray*}
Observing that each term of the above sum belongs to some $\ce_m$ for $t\le m\le t'-1$, we conclude the proof.
\end{proof}

Now we make a further assumption: Suppose that we can choose the stabilizing sequence such that, in the limit, the parameters 
$\eta_{m,j}$ are uniformly bounded (we require nothing on the limit parameters $\pi_m$). Recall that, in Section~\ref{sec:Intro}, we defined such a situation as the \emph{bounded-recurrent} case.

\begin{Lemma}
  \label{lemma:analyticity_of_P_j}
  In the bounded-recurrent case, for each $j\ge 1$ the sequence $\left(P_j(r)\right)_{r\ge0}$ converges to $0$ exponentially fast as $r\to\infty$.
\end{Lemma}
\begin{proof}
  Suppose that all coefficients $\eta_{m,j}$ are bounded by $R$. Then by~\eqref{eq:def_gamma}, $\gamma^{(j)}\ge r$ implies the existence of $\ell\in\{0,\ldots,j-1\}$ such that $\ov t(\ov S^\ell\ov y)\ge r/(jR)$, but this happens with probability less than $j\,2^{-r/(jR)}$.
  \end{proof}

Finally, to get our spectral disjointness result, we need a last assumption ensuring that the limit distribution $P_j$ is not concentrated on a single point. We say that the recurrent rank-one construction is \emph{non-flat} if we can choose the stabilizing sequence in such a way that there exists at least one $m\in\Z$ with $\ce_m\neq\{0\}$.

\begin{Th}
\label{thm:spectral-disjointness} Assume that the construction of the rank-one transformation $S^f$ is bounded-recurrent and non-flat. Then  for any $1\le j_1<j_2$ the continuous parts of the maximal spectral types of $(S^f)^{j_1}$ and $(S^f)^{j_2}$ are mutually singular.
\end{Th}

\begin{proof}
Let $j_1,j_2\ge 1$ and assume that the continuous parts of the maximal spectral types of $(S^f)^{j_1}$ and $(S^f)^{j_2}$ are not mutually singular.
 Let $d_\infty\egdef\gcd \left(\bigcup_{m\in\Z}\ce_m\right)$. We can find a finite family $\{ d_1,\ldots,d_\ell\}\subset\bigcup_{m\in\Z}\ce_m$ such that $d_\infty=\gcd(d_1,\ldots,d_\ell)$. Let $d\in\{d_1,\ldots,d_\ell\}$, and let $m_0$ be such that $d\in\ce_{m_0}$.
Note that for any integer $L$, the shifted subsequence $(n_k-L)_{k\ge 1}$ is also stabilizing, and that replacing $(n_k)$ by $(n_k-L)$ simply shifts by $L$ the indices of $(\pi_m)$, $(\eta_m)$ and $(\ce_m)$. Hence, taking a shifted stabilizing subsequence $(n_k-L)$ if necessary, we can always assume that $2^{m_0-1}>\max(j_1,j_2)$.

We know that the distribution of  $f^{(j_iq_{n_k})}-j_ih_{n_k+1}$ converges weakly to $P_{j_i}$ ($i=1,2$).
Since the functions $(\eta_m)$ are uniformly bounded, Lemma~\ref{lemma:analyticity_of_P_j} ensures that $P_{j_i}(r)$ decreases exponentially fast.
We can apply Proposition~\ref{mop1}, which gives the weak convergence of $(S^f)_{j_ih_n}$ to $F_{j_i}(S^f)$, where
$$ F_{j_i}(z) = \sum_{r\ge 0} P_{j_i}(r) z^r $$
is analytic in $\ov \bd$. We can now apply Corollary~\ref{clb2}: Since the continuous parts of the maximal spectral types of $(S^f)^{j_1}$ and $(S^f)^{j_2}$ are not mutually singular, we have
\begin{equation}
\label{eq:Fj_1Fj_2}
 j_1 \Supp(F_{j_2}) = j_2 \Supp(F_{j_1}),
\end{equation}
where $\Supp(F_{j_i}) = \{ r\in\Z:\ P_{j_i}(r)\neq0\}$. By Lemma~\ref{lemma:distribution_of_Pj}, we know that
$\Supp(F_{j_2})$ contains two integers whose difference equals $d$. On the other hand, Lemma~\ref{lemma:multiple_of_d} ensures that the difference between two elements of $\Supp(F_{j_1})$ is always a multiple of $d_\infty$. From~(\ref{eq:Fj_1Fj_2}), we then get that $j_1d_\infty$ divides $j_2d$. But this holds for any $d\in\{d_1,\ldots,d_\ell\}$, hence $j_1$ divides $j_2$, and by symmetry $j_1=j_2$.
\end{proof}

\subsection{Bounded rank-one constructions}
\label{sec:bounded rank one}

\begin{Prop}\label{th1}  Assume the existence of the stabilizing subsequence $(n_k)$ and that there exists at least one $m\in\Z$ such that $\ce_m\neq\{0\}$. Then the only possible eigenvalues are rational and there are only finitely many of them.\end{Prop}
\begin{proof}
 Replacing if necessary the stabilizing subsequence $(n_k)$ by $(n_k+m)$, we can assume that $\ce_0\neq\{0\}$. Then $\cs_0$ contains at least two different integers $\alpha<\beta$, and the function $\gamma$ defined on $\ov Y$ by (\ref{eq:def_gamma}) takes the values $\alpha$ and $\beta$ both with probability at least $1/(\pi_0\pi_1)$. Recall that the integral automorphism $S^f$ is a rank-one transformation which is given with a refining sequence of towers. Consider the distribution of the return time on the base of tower $n_k$: For all large enough $k$, this return time takes the values $h_{n_k}+\alpha$ and $h_{n_k}+\beta$ both with probability at least $1/(\pi_0\pi_1)$. Approximating an eigenvector by a function which is constant on the levels of tower $n_k$, Chacon's standard argument~\cite{C2} yields that any eigenvalue $\lambda$ of $S^f$ must satisfy
$$ \lambda^{h_{n_k}+\alpha} \tend{k}{\infty} 1,\quad\mbox{and } \lambda^{h_{n_k}+\beta} \tend{k}{\infty} 1. $$
In particular, $\lambda^{\beta-\alpha}=1$.
\end{proof}

\begin{Th}
\label{thm:bounded rank one}
 If in the construction of the rank-one transformation $S^f$ all parameters $p_n$ and $s_{n,j}$ are uniformly bounded, then one of the three following properties holds:
\begin{enumerate}
 \item $S^f$ is weakly mixing;
 \item $S^f$ has finitely many eigenvalues, and they are all rational;
 \item $S^f$ is isomorphic to an odometer, hence has discrete spectrum and rational eigenvalues.
\end{enumerate}
In either case different (positive) powers are spectrally disjoint on the continuous part of the spectrum.
Moreover, in the cases where eigenvalues exist, the corresponding eigenvectors are constant on levels of towers of sufficiently high order.
\end{Th}
\begin{proof}
Let us say that step $n$ of the construction is a \emph{flat step} if 
\begin{equation}
 \label{eq:flat_step}
s_{n,0}=\cdots=s_{n,p_n-1}=s_{n,p_n}+s_{n+1,0}=\cdots=s_{n,p_n}+s_{n+1,p_{n+1}-1}.
\end{equation}
Suppose that there exists an increasing sequence of $n$'s for which (\ref{eq:flat_step}) fails. Since all parameters are bounded, we can extract from this sequence a stabilizing subsequence for which $\ce_0\neq\{0\}$. Then, according to Proposition~\ref{th1}, either $S^f$ is weakly mixing or $S^f$ has finitely many eigenvalues which are all rational. In this case the spectral disjointness of the powers on the continuous part of the spectrum is given by Theorem~\ref{thm:spectral-disjointness}.

On the other hand, if there exists an integer $N$ such that, for any $n\ge N$, step $n$ is flat, then the return time on the base of tower $N$ is constant, equal to $h_{N}+s_{N,0}$. But then $S^f$ is isomorphic to the odometer constructed on the product space
$$ \{0,1,\ldots,h_N+s_{N,0}-1\}\times\prod_{n\ge N}\{0,1,\ldots,p_n-1\}. $$
In that case, $S^f$ has purely point spectrum, so there is nothing to prove concerning the spectral disjointness of the powers.

Now assume that $\lambda$ is a rational eigenvalue. As the cutting parameters are bounded, the adaptation of Chacon's approximation argument used in the proof of Proposition~\ref{th1} yields in fact that 
$$ \sup_{0\le j\le p_n-2} \left| \lambda^{h_n+s_{n,j}} -1\right| \tend{n}{\infty} 0. $$
But since $\lambda$ is rational, we get that for all $n$ large enough and all $0\le j\le p_{n}-2$,
\begin{equation}
  \label{eq:return_to_the_base}
\lambda^{h_n+s_{n,j}}  = 1.
\end{equation}
Let $k=\min\{\ell\ge1:\ \lambda^\ell=1\}$ be the order of $\lambda$.  Let $N$ be such that~(\ref{eq:return_to_the_base}) is valid for all $n\ge N$. Consider the base $B_N\egdef D_0^{(N)}\times\{0\}$ of tower $N$ for $T$: Then all return times to $B_N$ are divisible by $k$. Consider the function $f$ defined by 
$f=\lambda^j$ on $S^f_j B_N\setminus\bigcup_{0\le \ell\le j-1}S^f_\ell B_N$. It is straightforward to check that $f$ is an eigenvector associated to the eigenvalue $\lambda$, and that $f$ is constant on the levels of tower $N$. Since the eigenspace associated to $\lambda$ has dimension~1, this concludes the proof.
\end{proof}

\begin{Remark}
\em
If we drop the assumption that the cutting parameters of the rank-one construction are bounded, then eigenvectors associated to a rational eigenvalue may never be constant on levels of towers. 
\end{Remark}

Indeed, consider the following parameters: $p_n=2^n$, $s_{n,j}=1$ if $j=2^{n-1}-1$ or $j=2^{n-1}$, $s_{n,j}=0$ otherwise. Consider the function $f_n$ supported on the levels of tower $n$, taking the value $1$ on even levels and $-1$ on odd levels. It is easy to check that $f_{n+1}$ coincides with $f_n$ outside a set whose measure decreases exponentially fast. An application of Borel-Cantelli lemma shows that the sequence $(f_n)$ converges almost everywhere to a function $f$ which is an eigenvector for the eigenvalue $-1$. But this eigenvector is never constant on the levels of the towers since some return times to these levels are odd.

\section{Examples where the cutting parameter $p_n$ is unbounded}
\label{sec:unbounded}
The purpose of this section is to show that our method for proving the spectral disjointness of the powers can also work in classical rank-one examples where we have $p_n\tend{n}{\infty}\infty$, which of course prohibits the existence of stabilizing subsequences. The two classes of examples which we present here are known to be rigid, and weakly mixing (see~\cite{Fe}). 

\subsection{Rigid Generalized Chacon's maps}
The parameters of this family of rank-one constructions satisfy the following properties:
\begin{itemize}
  \item $p_n\tend{n}{\infty}\infty$;
  \item there exist $0\le r_n<p_n$ such that $s_{n,r_n}=1$ and $s_{n,j}=0$ for $j\neq r_n$. 
\end{itemize}
The same cocycle-over-odometer approach can be used to show that for any fixed $0<\alpha<1$, we have the following weak convergence (where $\lfloor x\rfloor$ stands for the integer part of the real number $x$)
\begin{equation}
  \label{eq:rigid-chacon}
  T^{-\lfloor\alpha p_n\rfloor h_n} \tend{n}{\infty} \alpha T + (1-\alpha) \Id.
\end{equation}
Now let $1\le j_1<j_2$, and choose $\alpha$ sufficiently small so that $j_2\alpha<1$. We get 
$$ T^{-j_i\lfloor\alpha p_n\rfloor h_n} \tend{n}{\infty} j_i\alpha T + (1-j_i\alpha) \Id \quad(i=1,2), $$
and by Proposition~\ref{llb1} we can conclude that $T^{j_1}\spdj T^{j_2}$. 

\subsection{Katok's map}

For this rank-one construction, we require $p_n$ to be an even number, growing sufficiently fast to infinity (to be precised below). The spacers are defined by
$$ s_{n,j}=\begin{cases}
             0 & \text{ if }0\le j<p_n/2,\\
             1 & \text{ if }p_n/2\le j<p_n.
           \end{cases}
$$
 Katok's maps appears as a special case of a three-interval exchange maps~\cite{Fe-Ho-Za}. We point out that Bourgain in~\cite{Bo} proved that all three-exchange maps  with Keane condition are disjoint from Möbius function.

In this example, analytic functions of $U_T$ abound in the weak closure of $\{U_T^k:\:k\in\Z\}$. Indeed, Ryzhikov \cite{Ry2} has shown that we can find as a weak limit of some subsequence $U_T^{n_k}$ any convex combination of $\Theta$ and the $U_T^k$, $k\in\Z$, where $\Theta$ denotes the orthogonal projection on the constant subspace (which is zero if we only consider the action of the Koopman operators on the subspace of $L^2$ orthogonal to constant functions). However, the limits described in~\cite{Ry2} are all derived from the following weak convergence:
$$ U_T^{-j h_n}\tend{n}{\infty} \dfrac{Id+U_T^j}{2}. $$
It follows that, if $ U_T^{\ell_n}\tend{n}{\infty} F(U_T)$ is a weak limit given by Ryzhikov's argument, then we also have $U_T^{j\ell_n}\tend{n}{\infty} F(U_T^j)$,
which is precisely the situation where, in spite of the existence of the weak limits, our method can not conclude that we have spectral disjointness of different powers. 

We therefore have to look for other weak limits. To simplify the argument, we make the following assumption on the growth rate of $p_n$:
\begin{equation}
  \label{eq:growth rate in Katok}  
  \dfrac{p_n}{h_n} \tend{n}{\infty} +\infty. 
\end{equation}
Let $0<\alpha<1$. By~\eqref{eq:growth rate in Katok}, we can find a sequence of integers $(\ell_n)$ such that for all $n$, 
\begin{align}
  \label{multiple} &\ell_n\text{ is a multiple of }h_n+1,\\
  \label{asymptotics of ell_n} &\ell_n=\alpha p_n/2 + o(p_n).
  \end{align}
\begin{Lemma}
\label{lemma:alpha weak mixing of Katok}
  For $\ell_n$ as above,  we have the following weak convergence.
\begin{equation}
  \label{eq:weak convergence in Katok}
  U_T^{-\ell_n h_n} \tend{n}{\infty} \alpha \Theta + (1-\alpha) \Id.
\end{equation}
\end{Lemma}
Transformations admitting this kind of weak limits are said to be $\alpha$-weakly mixing. The $\alpha$-weak mixing of Katok's map for any $\alpha$ is of course not a new result (it is a direct consequence of Ryzhikov's theorem), but what is important for our purposes is the appropriate control of the subsequence along which we get this $\alpha$-weak mixing. 
\begin{proof}[Sketch of proof of Lemma~\ref{lemma:alpha weak mixing of Katok}]
Unfortunately, the cocycle-over-odometer approach used in the other examples does not work so well here. Indeed the corresponding limit distributions put some mass at infinity, and this approach only allows to see the $(1-\alpha) \Id$ part of the limit. We thus need a more down-to-earth method. 

Let $B$ be a level of tower $n$ in the construction. We consider the first half $B_1$ of $B$, which has no spacer above in tower $n$ (in light grey on Figure~\ref{fig:katok}). 
In the cutting process, $B_1$ is cut into $p_n/2$ pieces, which are successive images of the first piece by the transformation $T^{h_n}$. Consider the image of $B_1$ by $T^{\ell_n h_n}$: The leftmost $(p_n/2-\ell_n)$ pieces of $B_1$ are still in $B_1$ after this transformation, while the $\ell_n$ remaining pieces have moved to the rightmost half of tower $n$ where the height is now $h_n+1$ (this second half being coverd by a single spacer). Therefore these pieces will uniformly spread over tower $n$, and by~\eqref{asymptotics of ell_n} they represent a proportion approximately equal to $\alpha$ of $B_1$.  Consider now the second half $B_2$ of $B$ (in dark grey on Figure~\ref{fig:katok}).
$B_2$ is also cut into $p_n/2$ pieces, but these pieces are now successive images by the transformation $T^{h_{n+1}}$. A similar destiny is reserved to these pieces when transformed by $T^{\ell_n h_n}$: The $(\ell_n h_n)/h_{n+1}$ rightmost pieces of $B_2$ move into the left half of 
tower $n$, where they uniformly spread over the levels of the tower, while the remaining pieces which represent a proportion approximately equal to $(1-\alpha)$ of $B_2$ are still in $B_2$ by~\eqref{multiple}.  

\begin{figure}[htp]
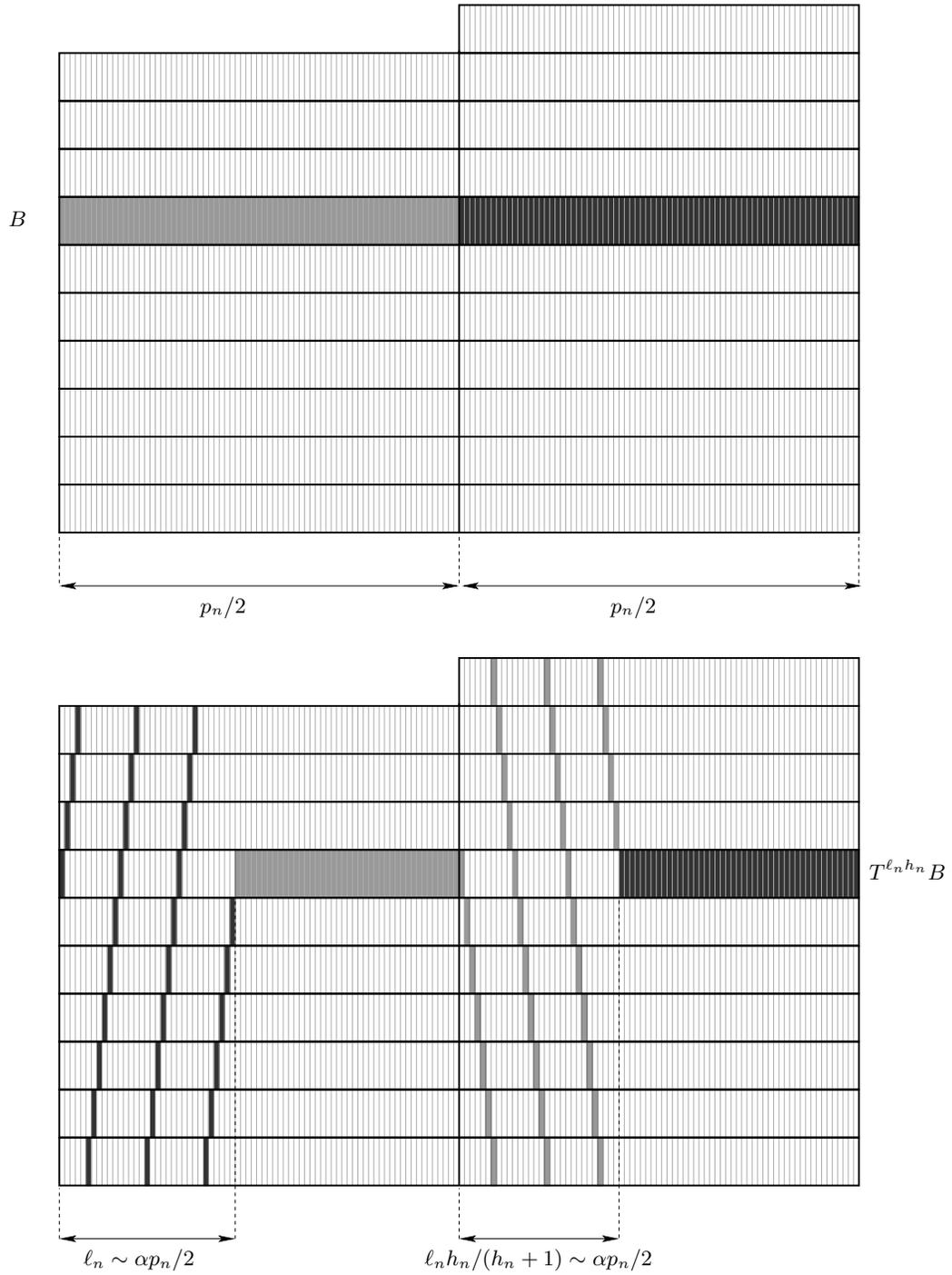

  \centering
  \input{katok1.pstex_t}\\
  
  \vspace{5mm}
  
  \input{katok2.pstex_t}
  \caption{$\alpha$ weak-mixing in Katok's map. Top: A level $B$ in tower $n$. Bottom: Image of $B$ by $T^{\ell_n h_n}$. Note that the picture is drawned as if there were no spacer above towers $n+1$, $n+2$\dots. These spacers only affect the staircase on the left of the bottom figure, but they do not change its uniform spreading over the levels of the tower. }
  \label{fig:katok}
\end{figure}

Now, the method to end the proof is quite standard. If $A$ and $B$ are unions of levels of some tower $n_0$, then they are also unions of levels of any tower $n$ for $n\ge n_0$, and the above analysis proves that 
$$ \mu\left(A\cap T^{\ell_n h_n} B\right) \tend{n}{\infty} \alpha \mu(A)\mu(B) + (1-\alpha)\mu(A\cap B). $$
This convergence extends to arbitrary measurable sets $A$ and $B$ since any set can be approximated by unions of levels of a tower in the construction. This proves the announced weak convergence when the operators act on simple functions, and this is enough to conclude by density of simple functions in $L^2$.
\end{proof}

With this result, we can now proceed as in the preceding example: Let $1\le j_1<j_2$, choose $\alpha$ small enough so that $j_2\alpha < 1$ and define $\ell_n$ as above. Then, by~\eqref{eq:weak convergence in Katok} we have for $i=1,2$
$$ 
U_T^{-j_i\ell_n h_n} \tend{n}{\infty} j_i\alpha \Theta + (1-j_i\alpha) \Id,
$$
which by Proposition~\ref{llb1} is enough to state the spectral disjointness of $T^{j_1}$ and $T^{j_2}$.

\begin{Remark}
 \emph{We can observe that the same argument applies for a family of random constructions described by Del~Junco and the second author in~\cite{DJ-L}.}
\end{Remark}

\section{Sarnak's conjecture for the symbolic model associated to some rank-one transformations}
  
  \label{sec:symbolic}

  \subsection{Symbolic model for a rank-one construction}
  We recall first the construction of the symbolic model associated to a rank-one automorphism given by the cutting and stacking method.  In order to distinguish the different models, we denote here by $\tcs$ the rank-one automorphism of $([0,1),\lambda)$ (where $\lambda$ is Lebesque measure) constructed by cutting and stacking from the parameters $(p_n)_{n\ge1}$ and $(s_{n,i})_{n\ge 1, 0\le i\le p_n-1}$. We consider the partition $\Pa=\{P_0,P_1\}$ of  $[0,1)$, where $P_0$ is the unique level of tower~1, and $P_1\egdef [0,1)\setminus P_0$ is the union of all spacers added in the successive steps of the construction.
   With the same parameters, we inductively define a sequence of finite words $(B_n)_{n\ge 1}$ over the alphabet $\{0,1\}$, which we call \emph{building blocks}:
  $$ B_1\egdef 0 ; \quad B_{n+1}\egdef B_n 1^{s_{n,0}}B_n 1^{s_{n,1}}\cdots B_n 1^{s_{n,p_n-1}}. $$
  The length $|B_n|$ of the building block of order $n$ is equal to $h_n$, the height of tower $n$, and we can view $B_n$ as the $\Pa$-name (of length $h_n$) of a point in the basis of tower~$n$. The symbols $1$ in the building blocks will also be called \emph{spacers}.
  
  Then we consider the subshift $\Omega\subset \{0,1\}^\Z$, which is the set of bi-infinite sequences $(\omega_j)_{j\in\Z}$ satisfying
  \begin{equation}
    \label{eq:def_subshift}
    \forall i<j ,\ \omega|_i^j\egdef \omega_i \omega_{i+1} \ldots \omega_{j-1} \mbox{ is a subword of }B_m\mbox{ for some $m\ge1$}.
  \end{equation}
We consider on $\Omega$ the product topology, which turns $\Omega$ into a metrizable compact space, and we denote by $\ts$ the shift of coordinates, which is a homeomorphism of $\Omega$. We recall that, since our rank-one construction is supposed to live on a probability space, we have
\begin{equation}
  \label{eq:finite_measure}
  \sum_{n\ge 1}\dfrac{1}{h_{n+1}} \sum_{i=0}^{p_n-1}s_{n,i}<\infty.
\end{equation}
From this, it follows that for any finite word $W$, the frequency of $W$ in $B_n$ converges as $n\to\infty$ to a limit $\mu(W)$. The family $\mu(W)$ when $W$ ranges over the set of all finite words defines a shift-invariant probability measure $\mu$ on $\Omega$.  The measure-preserving dynamical system $(\Omega, \ts, \mu)$ can now be viewed as a factor of the rank-one automorphism $\tcs$, via the map $\psi:\ [0,1)\to\Omega$ defined by 
$$ \psi(x) \egdef \bigl( \Pa(T^jx)\bigr)_{j\in\Z}. $$

As observed by Ferenczi in~\cite{Fe}, the symbolic system  $(\Omega, \ts)$ may in some cases be degenerate (for example with the von Neumann-Kakutani construction: $B_{n+1}=B_nB_n$). However, it follows from the analysis developed by Kalikow in the annex of~\cite{kalikow} that the factor map $\psi$ is one-to-one, hence an isomorphism of the measure-preserving systems, whenever the limit sequence $B_\infty\egdef \lim_{n\to\infty}B_n$ is aperiodic. 

\begin{Lemma}
  \label{lemma:periodicity}
  Assume that $B_\infty$ is periodic. Then $\tcs$ is isomorphic to an odometer, in particular, it has infinitely many rational eigenvalues.
\end{Lemma}
\begin{proof}
  Let $p$ be a period of $B_\infty$, and $n$ large enough so that $h_n>p$. Let $a$ be the number of terminal spacers in $B_n$, so that $B_n$ can be written $B_n'01^a$. Consider two successive occurrences of $B_n$ inside $B_\infty$, and let $s$ be the number of spacers between them. Then, since $B_n$ starts with a $0$, we can see inside $B_\infty$ the sequence $B_n'01^{a+s}0$. By periodicity of $B_\infty$, the pattern $01^{a+s}0$ appears at a definite position (depending only on $p$) inside $B'n$. It follows that $s$ is determined by $B'_n$, therefore the number of spacers between two successive occurrences of $B_n$ is always the same. In other words, the return-time to the basis of tower~$n$ is constant, and $\tcs$ is isomorphic to an odometer.
\end{proof}

As an immediate consequence, we get that the symbolic model is isomorphic to $\tcs$ provided that the latter has finitely many eigenvalues, which we henceforth assume. To lighten the notations, we now write $T$ for $\ts$.

\medskip

Let us define a metric $d$ associated to the topology of weak convergence on the space of all probability measures on $\Omega$: Let $(C_n)_{n\ge0}$ be the countable family of cylinder sets in $\{0,1\}^\Z$, and set
$$ d(\nu_1,\nu_2)\egdef \sum_{n\ge0}\dfrac{1}{2^n}\left|\nu_1(C_n)-\nu_2(C_n)\right|. $$
For $\omega\in\Omega$ and $\ell\ge1$, let $\eta_\ell(\omega)$ be the \emph{empirical} measure
$$ \eta_\ell(\omega)\egdef \dfrac{1}{\ell} \sum_{j=0}^{\ell-1} \delta_{T^j\omega}. $$
It is easy to show from the definition of $\mu$ that
\begin{equation}
  \label{eq:unif_conv_on_B_n}
  \sup_{\omega\in\Omega:\ \omega|_0^{h_n}=B_n} d\left(\eta_{h_n}(\omega),\mu\right) \tend{n}{\infty}0.
\end{equation}

The measure $\mu$ need not be the unique ergodic invariant probability on $(\Omega,T)$. Indeed, the sequence $\bun\egdef\ldots 111\ldots$ which is fixed by $T$ may belong to $\Omega$. This is of course the case if the parameters $s_{n,i}$ are unbounded, but this may also hold in the case of bounded parameters, if spacers accumulate on the last subcolumn of each tower (think of the historic construction of Chacon in which $B_{n+1}=B_nB_n1$). In any case, there are at most two ergodic invariant probability measures: $\mu$ and $\delta_{\bun}$. The only nonatomic invariant probability measure is always $\mu$.

If $\bun\not\in \Omega$, the symbolic model is uniquely ergodic, hence Sarnak's conjecture holds for this model as soon as $T^p$ and $T^q$ are disjoint for any different prime numbers $p$ and $q$. But if $\bun\in \Omega$, we have to work a little harder to see that Sarnak's conjecture is valid for the symbolic model. We will get the result at the cost of an additional hypothesis on the construction. 

\subsection{A combinatorial result for ``reasonable" rank-one constructions}

\begin{Def}
 \em We say that the rank-one construction is \emph{reasonable} if
\begin{equation}
  \label{eq:hyp_spacers}
 t_n\egdef\sup_{0\le i \le p_n-1} s_{n,i}=o(h_n)\quad\mbox{ as }n\to\infty.
\end{equation}
\end{Def}

We assume henceforth that the above assumption is satisfied. Using the fact that, for all $n\ge 1$, $h_{n+1}\ge 2h_n$, it easily follows that 
\begin{equation}
  \label{eq:sum of spacers}
  \dfrac{t_1+\cdots+t_n}{h_n}\tend{n}{\infty}0.
\end{equation}

\bigskip

Inside the word $B_m$, we distinguish several orders of spacers (that is, of occurrences of the symbol $1$). We say that a specific spacer inside $B_m$ is of order $n\le m$ if this spacer lies inside a copy of $B_n$ but not inside a copy of $B_{n-1}$. For example, if we take the classical Chacon transformation, with $B_{n+1}=B_nB_n1B_n$, we have 
$$ B_3=0010\,0010\, 1\, 0010, $$
where the third occurrence of the symbol $1$ is of order 3, whereas all others are of order $2$.

The following lemma is an easy consequence of the structure of the building blocks.

\begin{Lemma}
\label{lemma:0 of B_m}
  Let $1\le n<m$. Every symbol $0$ inside $B_m$ lies in a copy of $B_n$ entirely contained in $B_m$.
\end{Lemma}

\begin{Lemma}
\label{lemma:between spacers}
Let $1\le n<m$. If two spacers inside $B_m$ are both of order $\ge n+1$, and if there exists between them a symbol $0$, then we can find between them at least one occurrence of the building block $B_n$.  
\end{Lemma}
\begin{proof}
  This is a direct consequence of Lemma~\ref{lemma:0 of B_m}.
\end{proof}

We will also need the following result to bound the number of consecutive spacers inside a building block.

\begin{Lemma}
\label{lemma:spacers inside B_(n+1)}
  Let $n\ge 1$, and $1^s$ a block of $s$ consecutive spacers inside $B_{n+1}$. Then 
  $$ s\le t_1+\cdots+t_n. $$
\end{Lemma}
\begin{proof}
  By the definition of $t_n$, the length of a block of consecutive spacers of order $(n+1)$ is bounded by $t_n$. Such a block may be adjacent to a block of consecutive spacers at the end of $B_n$, and the proof follows by an easy induction on $n$. 
\end{proof}

We are now ready to state and prove the following key proposition on the structure of words in the language of the subshift $\Omega$.

\begin{Prop}
\label{prop:structure of names}
  Fix a real number $\varepsilon>0$ and an integer $\ell\ge1$. Then there exists $N(\varepsilon,\ell)$ such that, for all integer $N\ge N(\varepsilon,\ell)$ and all $\omega\in\Omega$, the word $\omega|_1^N$ can be represented in the form 
  $$ \omega|_1^N = ABC, $$
  where the (possibly empty) words $A$, $B$, $C$ satisfy the following:
  \begin{itemize}
    \item $B=1^s$ for some $s\ge 0$,
    \item we can cover a \emph{large} part of $A$ and $C$ with disjoint building blocks of order $\ge \ell$, where ``large" means that the total number of letters in $A$ and $C$ which are not covered is bounded by $\varepsilon N$.
  \end{itemize}
\end{Prop}

\begin{proof}
Replacing if necessary $\ell$ by a larger integer, we may assume by~\eqref{eq:sum of spacers} that 
\begin{equation}
  \label{eq:ell large enough}
  \sup_{n\ge \ell} \dfrac{t_1+\cdots+t_n}{h_n} < \varepsilon/4. 
\end{equation}
Now, fix $\omega\in\Omega$, $N\ge1$, and consider the word $W\egdef \omega|_1^N$. 
We know by the definition of $\Omega$ that $W$ is a subword of the building block $B_m$ for some $m\ge0$. In general, we can see several copies of $W$ inside $B_m$, but we consider one particular occurrence of $W$ in $B_m$. By the inductive construction of the building blocks, for any $1\le n\le m$, $B_m$ is canonically decomposed into building blocks $B_n$ and spacers of order at least $n+1$.  When considering building blocks $B_n$ inside $W$, we will implicitly assume that these building blocks come from this canonical decomposition. And since we have fixed one particular instance of $W$ in $B_m$, any spacer in $W$ also has a well-defined order $n\le m$, which is inherited from the corresponding spacer inside $B_m$.

It might happen that we cannot find in $W$ any building block $B_\ell$. By Lemma~\ref{lemma:0 of B_m}, this implies that any symbol $0$ inside $W$ is either in the first $h_\ell-1$ letters of $W$ or in the last $h_\ell-1$ letters of $W$. Hence, in this case we can write $W$ in the form
$$ W = A \, 1^s \, C, $$
where $|A|<h_\ell$, $|C|<h_\ell$, and $s\ge 0$. The proposition easily follows in this case, by taking $N$ large enough so that $h_\ell/N<\varepsilon/2$.

\medskip

Otherwise, let $n_0\ge\ell$ be the largest order such that we can see at least one copy of $B_{n_0}$ inside $W$. Taking as many copies of $B_{n_0}$ as we can inside $W$, we decompose $W$ in the form
$$ W=W_{n_0}^S\,  W_{n_0}\, W_{n_0}^P, $$
where $W_{n_0}$ is covered by copies of $B_{n_0}$ and spacers of order $\ge n_0+1$, $W_{n_0}^S$ is a (possibly empty) suffix of $B_{n_0}$, and $W_{n_0}^P$ is a (possibly empty) prefix of $B_{n_0}$. Then, we consider inside 
$W_{n_0}^S$ and $W_{n_0}^P$ all possible copies of $B_{n_0-1}$. This yields a decomposition of the form
$$ W=W_{n_0-1}^S\, W_{n_0-1}\, W_{n_0}\, W'_{n_0-1}\, W_{n_0-1}^P, $$
where $W_{n_0-1}$ and $W'_{n_0-1}$ are covered by copies of $B_{n_0-1}$ and spacers of order $n_0$, and $W_{n_0-1}^S$  (respectively $W_{n_0-1}^P$) is a suffix (respectively prefix) of $B_{n_0-1}$. 

Going on in the same way and considering successively the building blocks $B_{n_0-2}$, $B_{n_0-3},\ldots$, $B_\ell$, we finally get a decomposition of the form
$$ W=W_\ell^S \, W_\ell \, W_{\ell+1} \ldots W_{n_0-1}\, W_{n_0}\, W'_{n_0-1} \ldots W'_{\ell+1} \, W'_\ell \, W_\ell^P, $$
where
\begin{itemize}
  \item $W_{n_0}$ is covered by copies of $B_{n_0}$ and spacers of order $\ge n_0+1$,
  \item for each $\ell\le n\le n_0-1$, $W_n$ and $W'_n$ are covered by copies of $B_n$ and spacers of order $n+1$,
  \item $W_\ell^S$ is a suffix of $B_\ell$,
  \item $W_\ell^P$ is a prefix of $B_\ell$.
\end{itemize}
Note that, in this decomposition, $W_n$  ($\ell\le n\le n_0-1$) may be empty,  and it may also be reduced to a single block of spacers of order $n+1$.  In this latter case, its length is bounded by $t_n$. (The same holds for $W'_n$.)

Putting aside $W_\ell^S$ and $W_\ell^P$, whose lengths are always bounded by $h_\ell$, the rest of $W$ can be viewed as a concatenation of building blocks of order $\ge \ell$, separated by blocks of consecutive spacers. Let us say that a block of consecutive spacers is \emph{huge} if it contains a spacer of order $\ge n_0+2$. By Lemma~\ref{lemma:between spacers}, there can exist at most one huge block of spacers (otherwise we could have found in $W$ a building block of order $\ge n_0+1$).  

We claim that the sum of all lengths of blocks of spacers which are not huge represents a small fraction of $N$. 

Indeed, consider a copy of the building block $B_n$ appearing in some $W_n$, $\ell\le n\le n_0-1$. By the construction of the decomposition of $W$, the word $W_\ell^S \, W_\ell \, W_{\ell+1} \ldots W_{n}$ is subword of $B_{n+1}$. We then get by Lemma~\ref{lemma:spacers inside B_(n+1)} that the length of the block of spacers immediately to the left of our copy of $B_n$ is  at most $t_{1}+\cdots+t_n$, which is small with respect to $h_n=|B_n|$ by~\eqref{eq:ell large enough}. In the same way, if $B_n$ appears in some $W'_n$, $\ell\le n\le n_0-1$, then the length of the block of spacers immediately to its right is bounded by the same quantity. Now, consider a copy of $B_{n_0}$ appearing in $W_{n_0}$. If the block of spacers immediately to its left is not the huge one, then 
\begin{itemize}
  \item the part of this block of spacers contained inside $W_{n_0}$ has a length bounded by $t_{n_0}$,
  \item the possible part of this block of spacers contained inside $W_\ell^S \, W_\ell  \ldots W_{n_0-1}$ has, again by  Lemma~\ref{lemma:spacers inside B_(n+1)}, a length bounded by $t_1+\cdots+t_{n_0-1}$. 
\end{itemize}
The same applies to  the block of spacers immediately to its right: If this block of spacers is not the huge one, then its length is bounded by $t_1+\cdots+t_{n_0}$, which is small compared to $h_{n_0}$ by~\eqref{eq:ell large enough}. 
We finally get that the sum of the lengths of all non-huge blocks of spacers separating the building blocks in the decomposition of $W$ is bounded by
$$ 2N\sup_{n\ge \ell}\dfrac{t_1+\cdots+t_n}{h_n}, $$
which is smaller than $N \varepsilon/2$ by~\eqref{eq:ell large enough}. Therefore, the proposition follows by taking $N$ so large that $h_\ell/N<\varepsilon/4$, and $B$ as the possible huge block of spacers.  
\end{proof}

\subsection{Sarnak's conjecture for symbolic models of rank-one automorphisms}

\begin{Th}
\label{thm:Sarnak for symbolic WM}
  Let $(\Omega,T)$ be the symbolic model for a reasonable, weakly mixing, rank-one construction. Let $\mu$ be the unique non-atomic $T$-invariant probability measure. Assume that for all different prime numbers $p$ and $q$, $(\Omega,T^p,\mu)$ and $(\Omega,T^q,\mu)$ are disjoint. Then Sarnak's conjecture is valid for $(\Omega,T)$.
\end{Th}

\begin{proof}
  Let $f\in C(\Omega)$ and $\omega_0\in\Omega$. We have to prove the orthogonality of the sequence $\bigl(f(T^n\omega_0)\bigr)_{n\ge1}$ with the Möbius function. Since $\sum_{n\le N} \mob(n)=o(N)$, we can assume without loss of generality that $\int_\Omega f\,d\mu=0$.  
  
  As we have already noticed, if $T$ is uniquely ergodic the result is a straightforward consequence of Bourgain-Sarnak-Ziegler's criterion. 
  But if $\bun\in \Omega$ (which we assume henceforth), there exists a second ergodic invariant measure $\delta_{\bun}$, and the criterion does not apply directly. As it will appear more clearly in the further part of the proof, the value of $f$ at the fixed point $\bun$ may play an important role in the estimations, and the situation is easier if $f(\bun)=0$.
  Unfortunately, substracting the value $f(\bun)$ from $f$ kills the centering assumption on $f$, which is essential in the joining argument. We have therefore to find a trade-off between the centering assumption and the annihilation of $f(\bun)$. 
  
  For all $k\ge 1$, let $G_k$ be the cylinder set 
  $$ G_k \egdef \{\omega\in\Omega:\ \omega_i=1,\ 0\le i\le k-1\}. $$
  Of course, $\bun\in G_k$ for all $k$, and we have $\mu(G_k)\tend{k}{\infty}0$. Now, we set $f_k\egdef f-f(\bun)\ind{G_k}$, so that 
  \begin{itemize}
    \item $f_k$ remains continuous,
    \item $f_k(\bun)=0$,
    \item $|\int_\Omega f_k\,d\mu| = |f(\bun)|\, \mu(G_k) \tend{k}{\infty}0$.
  \end{itemize}
  
  We will apply Bourgain-Sarnak-Ziegler's criterion to $f_k$ and $\omega_0$ in the following form: It is proved in \cite{Bo-Sa-Zi} that, if for some $\varepsilon>0$ and for all different prime numbers $p$ and $q$ less than $\exp(1/\varepsilon)$, we have
  \begin{equation}
    \label{eq:f_k limit}
    \limsup_{N\to\infty}\dfrac{1}{N} \left| \sum_{n=1}^N f_k (T^{pn}\omega_0) f_k(T^{qn}\omega_0) \right| < \varepsilon,
  \end{equation}
  then 
  \begin{equation}
    \label{eq:almost orthogonality for f_k}
    \limsup_{N\to\infty}\dfrac{1}{N} \left| \sum_{n=1}^N \mob(n) f_k(T^{n}\omega_0) \right| < 2\sqrt{\varepsilon\log1/\varepsilon}.
  \end{equation}
  
  To get~\eqref{eq:f_k limit}, it is enough to show that if $\rho$ is the weak limit of a subsequence of the sequence of empirical probability measures
  $$ \dfrac{1}{N} \sum_{n=1}^N \delta_{\left(T^{pn}\omega_0,T^{qn}\omega_0\right)}, $$
  then 
  $$ \left| \int_{\Omega\times \Omega} f_k(\omega_1) f_k(\omega_2) \, d\rho(\omega_1,\omega_2) \right| < \varepsilon. $$
  So, let $\rho$ be such a probability measure on $\Omega\times \Omega$. Then $\rho$ is $T^p\otimes T^q$-invariant. Denote by $\rho_1$ (respectively $\rho_2$) its marginal distribution on the first (respectively second) coordinate. The probability measure $\frac{1}{p} \sum_{i=0}^{p-1} \rho_1 \circ T^i $ is $T$-invariant on $\Omega$, hence there exists $\theta\in[0,1]$ such that 
  \begin{equation}
    \label{eq:T-invariant proba}
    \dfrac{1}{p} \sum_{i=0}^{p-1} \rho_1 \circ T^i = \theta \delta_{\bun} + (1-\theta)\mu.
  \end{equation}

  Set $\Omega'\egdef \Omega\setminus\{\bun\}$. Let us decompose $\rho$ as
  \begin{multline}
  \label{eq:decomposition of rho}
    \rho = \rho(\Omega'\times \Omega')\,\rho(\,\cdot\,|\Omega'\times \Omega') + \rho(\Omega'\times \{\bun\})\,\rho(\,\cdot\,|\Omega'\times \{\bun\})\\ +\rho(\{\bun\}\times \Omega')\,\rho(\,\cdot\,|\{\bun\}\times \Omega') +\rho(\{(\bun,\bun)\})\,\rho(\,\cdot\,|\{(\bun,\bun)\}). 
  \end{multline}
  Here, $\rho(\,\cdot\,|\{(\bun,\bun)\})$ can not be anything else than $\delta_{\{(\bun,\bun)\}}$.
  Assume that $\rho(\Omega'\times \{\bun\})>0$. Then, we see by~\eqref{eq:T-invariant proba} that the first marginal of $\rho(\,\cdot\,|\Omega'\times \{\bun\})$ is absolutely continuous with respect to $\mu$. But it is also $T^p$-invariant, and since $\mu$ is $T^p$-ergodic, this first marginal has to be $\mu$. Therefore, in~\eqref{eq:decomposition of rho}, we can replace $\rho(\,\cdot\,|\Omega'\times \{\bun\})$ by $\mu\otimes\delta_{\bun}$. By a similar argument, we can replace $\rho(\,\cdot\,|\{\bun\}\times \Omega')$ by $\delta_{\bun}\otimes\mu$. In the same way, we also prove that, if $\rho(\Omega'\times \Omega')>0$, then both marginals of $\rho(\,\cdot\,|\Omega'\times \Omega')$ are equal to $\mu$. It follows that $\rho(\,\cdot\,|\Omega'\times \Omega')$ is a joining of $(\Omega,T^p,\mu)$ and $(\Omega,T^q,\mu)$. But we have assumed that these two systems are disjoint, hence this measure must be the product measure $\mu\otimes\mu$.
  Finally, \eqref{eq:decomposition of rho} becomes
  \begin{multline}
  \label{eq:decomposition of rho bis}
    \rho = \rho(\Omega'\times \Omega')\,\mu\otimes\mu + \rho(\Omega'\times \{\bun\})\,\mu\otimes\delta_{\bun} \\ +\rho(\{\bun\}\times \Omega')\, \delta_{\bun}\otimes\mu+\rho(\{(\bun,\bun)\})\,\delta_{\{(\bun,\bun)\}}. 
  \end{multline}
 
  Remembering that $f_k(\bun)=0$, it follows that 
  $$ \left| \int_{\Omega\times \Omega} f_k(\omega_1) f_k(\omega_2) \, d\rho(\omega_1,\omega_2) \right| \le \left( \int_\Omega f_k\, d\mu\right)^2, $$
  which goes to $0$ as $k\to\infty$. Hence, \eqref{eq:almost orthogonality for f_k} is valid for all large enough $k$.
  
  Let us now come back to $f$. For all $k\ge1$, we have
  $$ \dfrac{1}{N} \left| \sum_{n=1}^N \mob(n) f(T^{n}\omega_0) \right| \le \dfrac{1}{N} \left| \sum_{n=1}^N \mob(n) f_k(T^{n}\omega_0) \right| + \dfrac{|f(\bun)|}{N} \left| \sum_{n=1}^N \mob(n) \ind{G_k}(T^{n}\omega_0) \right|,
  $$
  and it only remains to estimate the second term in the RHS for large $k$. 
  
  Consider $k$ so large that $\mu(G_k)<\varepsilon$. Then, by~\eqref{eq:unif_conv_on_B_n}, we can fix $\ell$ large enough so that, for each $n\ge \ell$ and all $j$,
  \begin{equation}
    \label{eq:on building blocks}
    \omega|_j^{j+h_n}=B_n \Longrightarrow \sum_{i=j}^{j+h_n-1} \ind{G_k}(T^i\omega) < \varepsilon\, h_n.
  \end{equation}
  Then, we apply Proposition~\ref{prop:structure of names}: For $N$ large enough, we can find $1\le N_1\le N_2\le N$ such that 
  \begin{itemize}
    \item $\omega_0|_0^{N_1}$ and $\omega_0|_{N_2}^N$ are covered by building blocks of order $\ge \ell$ up to $\varepsilon\, N$ symbols,
    \item $\omega_0|_{N_1}^{N_2}=1^{N_2-N_1}$. 
  \end{itemize}
Then, by~\eqref{eq:on building blocks}, we get
$$ \left| \sum_{n=1}^N \mob(n) \ind{G_k}(T^n\omega_0) \right| \le 2N\,\varepsilon + k + 
   \left| \sum_{N_1\le n\le N_2-1-k} \mob(n)\right|. $$ 
The conclusion follows, since we can easily derive from $\left|\sum_{n\le N} \mob(n)\right|=o(N)$ that 
$$ \sup_{1\le N_1\le N_2\le N} \left| \sum_{N_1\le n\le N_2} \mob(n) \right| = o(N). $$
\end{proof}

Of course, Theorem~\ref{thm:Sarnak for symbolic WM} applies in the case of a weakly mixing bounded rank-one construction, for which we have proved that the spectral disjointness of different positive powers holds. But as we have seen in Theorem~\ref{thm:bounded rank one}, such a bounded construction may also have a finite number of eigenvalues (we exclude here the odometer case, where the symbolic model is not well defined). If non-trivial eigenvalues exist, Sarnak's conjecture for the symbolic model raises new difficulties in the non-uniquely ergodic case (that is, when $\bun\in\Omega$). Indeed, eigenfunctions are then discontinuous at the fixed point $\bun$. This is the reason why, before coming back to the symbolic model of rank-one transformations with a finite number of eigenvalues, we introduce an alternative topological model in which eigenfunctions are continuous. 

\medskip

Let $(\Omega,T)$ be the symbolic model of a weakly mixing, reasonable rank-one construction, and let $K\ge2$ be an integer. Consider the topological dynamical system $(\tilde\Omega, \tilde T)$ where 
  \begin{itemize}
    \item $\tilde\Omega\egdef\Omega\times\{0,\ldots,K-1\}$,
    \item $\tilde T(\omega,i)\egdef (\omega, i+1)$ if $i<K-1$, and $\tilde T(\omega, K-1)\egdef(T\omega,0)$.
  \end{itemize}
Let $\nu_K$ be the uniform distribution on $\{0,\ldots,K-1\}$, and $\tilde\mu\egdef \mu\otimes\nu_K$. Then $\tilde\mu$ is the unique nonatomic ergodic $\tilde T$-invariant probability measure on $\tilde\Omega$, and if $\bun\in\Omega$, there exists a second ergodic $\tilde T$-invariant probability measure which is $\delta_{\bun}\otimes\nu_K$. 

As an ergodic measurable dynamical system, $(\tilde\Omega,\tilde T,\tilde \mu)$ is rank one. It is not weakly mixing: It possesses $K$ eigenvalues $\varepsilon_K^j$, $0\le j\le K-1$, and the corresponding eigenfunctions are constant on each $\Omega_i\egdef\Omega\times\{i\}$, $0\le i\le K-1$ (in particular, they are continuous). 

\begin{Lemma}
  \label{lemma:joinings of tilde Omega}
  Assume that $1\le p<q$ are such that $\tilde T^p$ and $\tilde T^q$ are spectrally disjoint on the continuous part of the spectrum. Then any joining $\rho$ of $(\tilde\Omega,\tilde T^p,\tilde \mu)$ and $(\tilde\Omega,\tilde T^q,\tilde \mu)$ satisfies, for $f,g\in L^2(\tilde\mu)$
  $$ 
   \int_{\tilde \Omega\times \tilde\Omega} f(\tilde \omega_1) \, g(\tilde \omega_2)\,d\rho(\tilde \omega_1,\tilde \omega_2) = \sum_{0\le i,j\le K-1} \rho(\Omega_i\times\Omega_j)\, K^2\int_{\Omega_i}f\,d\tilde\mu \int_{\Omega_j}g\,d\tilde\mu.
  $$
\end{Lemma}
\begin{proof}
  Let us denote by $\ck$ the partition of $\tilde\Omega$ into subsets $\Omega_i$, $0\le i\le K-1$. Then $\ck$ yields the Kronecker factor of $(\tilde\Omega,\tilde T,\tilde \mu)$, and if $f\in L^2(\tilde\mu)$, 
  $$ f-\E_{\tilde \mu}[f|\ck] = f-\sum_{0\le i\le K-1} \ind{\Omega_i} K\int_{\Omega_i}f\,d\tilde\mu $$
  has a continuous spectral measure (for $\tilde T$, $\tilde T^p$ and $\tilde T^q$). Let $\rho$ be a joining of $(\tilde\Omega,\tilde T^p,\tilde \mu)$ and $(\tilde\Omega,\tilde T^q,\tilde \mu)$, and let $f,g\in L^2(\tilde\mu)$. It follows from the assumption of the lemma that  the spectral measure of the functions $(\tilde\omega_1,\tilde\omega_2)\mapsto f(\omega_1)-\E_{\tilde \mu}[f|\ck](\tilde \omega_1)$ and $(\tilde\omega_1,\tilde\omega_2)\mapsto g(\omega_2)-\E_{\tilde \mu}[g|\ck](\tilde \omega_2)$ under the action of $\tilde T^p\otimes \tilde T^q$ are continuous and mutually singular, hence orthogonal in $L^2(\rho)$. We then get
  \begin{align*}
    \int_{\tilde \Omega\times \tilde\Omega} f(\tilde \omega_1) \, g(\tilde \omega_2)\,d\rho(\tilde \omega_1,\tilde \omega_2) 
  & = \int_{\tilde \Omega\times \tilde\Omega}\, \E_{\tilde \mu}[f|\ck](\tilde \omega_1) \E_{\tilde \mu}[g|\ck](\tilde \omega_2)\,d\rho(\tilde \omega_1,\tilde \omega_2) \\
  & = \sum_{0\le i,j\le K-1} \rho(\Omega_i\times\Omega_j)\, K^2\int_{\Omega_i}f\,d\tilde\mu \int_{\Omega_j}g\,d\tilde\mu.
  \end{align*}

\end{proof}

\begin{Th}
\label{thm:Sarnak for non WM}
  Let $(\Omega,T)$ be the symbolic model of a weakly mixing, reasonable rank-one construction. Fix $K\ge2$ and consider $\tilde\Omega$,  $\tilde T$ and $\tilde\mu$ as above. Assume that for all distinct prime numbers $p,q$, $\tilde T^p$ and $\tilde T^q$ are spectrally disjoint on the continuous part of the spectrum. 
Then Sarnak's conjecture is valid for $(\tilde\Omega,\tilde T)$.
\end{Th}

\begin{proof}
   
  We will adapt the proof of Theorem~\ref{thm:Sarnak for symbolic WM} to this setting. We fix $f\in C(\tilde\Omega)$ and $(\omega_0,i_0)\in\tilde\Omega$, and we want to prove the orthogonality of the sequence $\bigl(f(\tilde T^n(\omega_0,i_0)\bigr)_{n\ge1}$ with the Möbius function. For $i\in\{0,\ldots,K-1\}$, let $\Omega_i\egdef\Omega\times\{i\}\subset\tilde\Omega$.  We may, without loss of generality, assume that 
  \begin{equation}
    \label{eq:integrale sur Omega_i}
    \forall i\in\{0,\ldots,K-1\},\quad \int_{\Omega_i} f\,d\tilde\mu=0.
  \end{equation}
 Indeed, otherwise set $a_i\egdef\int_{\Omega_i} f\,d\tilde\mu$ and replace $f$ with $f'\egdef f-\sum_{0\le i\le K-1}a_i\, \ind{\Omega_i}$ (which amounts to substracting the discrete-spectrum part of $f$).  Then the sequence $\bigl(f(T^n\omega_0)\bigr)_{n\ge 1}$ only differs from $\bigl(f'(T^n\omega_0)\bigr)_{n\ge 1}$ by the $K$-periodic sequence $(a_{i_0+n})_{n\ge 1}$ (where $i_0+n$ is computed modulo $K$), now orthogonality between a $K$-periodic function and the Möbius function is already known (see {\it e.g.} \cite{Ap}, Chapter~7).

  If $\bun\in\Omega$, we cannot deal directly with $f$ and we therefore choose a large integer $k$ and consider the continuous function $f_k$ defined by
  $$ f_k\egdef f - \sum_{0\le i\le K-1}  f(\bun,i)\,\ind{G_k\times\{i\}}, $$
  where $G_k$ is defined as in the proof of Theorem~\ref{thm:Sarnak for symbolic WM}. Then $f_k$ satisfies 
  \begin{itemize}
    \item $f_k(\bun,i)=0$ for all $i\in\{0,\ldots,K-1\}$,
    \item $f_k(\omega,i)=f(\omega,i)$ whenever $\omega\notin G_k$.
  \end{itemize}

  Let $p\neq q$ be distinct prime numbers, and let $\rho$ be the weak limit of a subsequence of the sequence of empirical probability measures
  $$ \dfrac{1}{N} \sum_{n=1}^N \delta_{\left(\tilde T^{pn}(\omega_0,i_0),\tilde T^{qn}(\omega_0,i_0)\right)}. $$
  We want to bound
  $$ \left| \int_{\tilde\Omega\times \tilde\Omega} f_k(\omega_1,i_1) f_k(\omega_2,i_2) \, d\rho(\omega_1,i_1,\omega_2,i_2) \right|. $$
  Note that we can assume that $p>K$ and $q>K$, which ensures the ergodicity of $(\tilde\Omega, \tilde \mu, \tilde T^p)$ and $(\tilde\Omega, \tilde \mu, \tilde T^p)$. Indeed, it is remarked in \cite{Bo-Sa-Zi} that, provided $\varepsilon$ is small enough, \eqref{eq:almost orthogonality for f_k} still holds if \eqref{eq:f_k limit} fails only for small prime numbers. 
  We decompose $\rho$ in the same way as in~\eqref{eq:decomposition of rho}, but here $\Omega'$ has to be replaced by $\tilde\Omega'\egdef(\Omega\setminus\{\bun\})\times\{0,\ldots,K-1\}$, and $\{\bun\}$ has to be replaced by $\bun\times\{0,\ldots,K-1\}$. Since $f_k(\bun,i)=0$ for all $i\in\{0,\ldots,K-1\}$, only the part $\rho(\,\cdot\,|\tilde\Omega'\times\tilde\Omega')$ will contribute to the integral. 
  But the same argument as in the proof of Theorem~\ref{thm:Sarnak for symbolic WM} gives that $\rho'\egdef\rho(\,\cdot\,|\tilde\Omega'\times \tilde\Omega')$ is a joining of $(\tilde \Omega,\tilde T^p,\tilde \mu)$ and $(\tilde \Omega,\tilde T^q,\tilde \mu)$ (here we use the ergodicity of $(\tilde\Omega, \tilde \mu, \tilde T^p)$ and $(\tilde\Omega, \tilde \mu, \tilde T^p)$). By Lemma~\ref{lemma:joinings of tilde Omega}, we are left with a finite sum of products of integrals of the form
  $$ \int_{\Omega_i}f_k\,d\tilde\mu \int_{\Omega_j} f_k\,d\tilde\mu, $$ 
  which, thanks to \eqref{eq:integrale sur Omega_i}, can be made arbitrarily close to 0 by choosing $k$ large enough.
  
  To come back to our original function $f$, and to conclude as in the proof of Theorem~\ref{thm:Sarnak for symbolic WM}, it remains to bound 
  \begin{multline*}
    \dfrac{1}{N} \left| \sum_{n=1}^N \mob(n) \sum_{0\le i\le K-1}\ind{G_k\times\{i\}}(\tilde T^{n}(\omega_0,i_0) f(\bun,i) \right| \\
    = \dfrac{1}{N} \left| \sum_{n=1}^N \mob(n) \ind{G_k\times\{0,\ldots,K-1\}}(\tilde T^{n}(\omega_0,i_0) f(\bun,i_0+n) \right|,
  \end{multline*}
  where the sum $i_0+n$ must be understood modulo $K$. 
  By application of Proposition~\ref{prop:structure of names}, we are reduced to estimate a sum
  $\left| \sum_{N_1\le n\le N_2} \mob(n) \varphi(n) \right|, $
  where $\varphi:\ n\mapsto f(\bun,i_0+n)$ is a $K$-periodic function, and $1\le N_1\le N_2\le N$.
  But again, since we know that $\left| \sum_{n\le N} \mob(n) \varphi(n) \right| =o(N)$, we get
  $$ \sup_{1\le N_1\le N_2\le N} \left| \sum_{N_1\le n\le N_2} \mob(n) \varphi(n) \right|=o(N). $$
  \end{proof}

\begin{Th}
\label{thm:Sarnak for bounded rank one}
  Sarnak's conjecture holds for the symbolic model of any bounded rank-one construction which is not an odometer. 
\end{Th}
\begin{proof}
Let $(\Omega',T')$ be the symbolic model of such a bounded rank-one construction, and $\mu'$ the nonatomic ergodic invariant measure. Denote by $(p'_n)$ and $(s'_{n,j})$, $n\ge1$, $0\le j\le p'_n-1$ the parameters of the construction, $(B'_n)$ the building blocks, and $(h'_n)$ their respective lengths. 

We know by Theorem~\ref{thm:spectral-disjointness} that for any $1\le p<q$ the continuous parts of the maximal spectral types of $(T')^p$ and $(T)'^q$ are mutually singular.
  By Theorem~\ref{thm:bounded rank one}, either the dynamical system $(\Omega',T',\mu')$ is weakly mixing and Theorem~\ref{thm:Sarnak for symbolic WM} applies, or it possesses a finite number of eigenvalues of the form $\varepsilon_K^j$ for some $K\ge 2$. In the latter case, the proof of Theorem~\ref{thm:bounded rank one} shows that, for $n$ large enough, the return time to the basis of tower~$n$ is always a multiple of $K$. In other words, there exists $n_0$ such that, for any $n\ge n_0$, two successive (canonical) occurrences of $B'_n$ have their starting positions separated by a multiple of $K$. Appending if necessary some extra spacers to the building blocks, but without changing the subshift $\Omega'$, we can always assume that for all $n\ge n_0$, $h'_{n}$ is a multiple of $K$. Then the parameters $(s'_{n,j})$  are themselves multiple of $K$ for all $n\ge n_0$. 
  
  Now let us define a new family of building blocks $(B_n)_{n\ge n_0}$. Set $h_{n_0}\egdef h'_{n_0}/K$, and $B_{n_0}\egdef 0^{h_{n_0}}$. Then, for all $n\ge n_0$ and all $0\le j\le p'_n-1$, set $p_n\egdef p'_n$, $s_{n,j}\egdef  s'_{n,j}/K$ and define inductively
  $$ B_{n+1} \egdef B_n 1^{s_{n,0}} B_n 1^{s_{n,1}} \cdots B_n 1^{s_{n,p_n-1}}. $$
  Observe that the length $h_n$ of $B_n$ is equal to $h'_n/K$. 
  
  These new building blocks $B_n$ define the symbolic model $(\Omega,T)$ associated to a bounded rank-one construction. Consider the system $(\tilde \Omega, \tilde T, \tilde \mu)$ constructed from $(\Omega,T)$ as in Theorem~\ref{thm:Sarnak for non WM}. Define the map $\varphi:\ \tilde\Omega\to\Omega'$ as follows: 
  \begin{itemize}
    \item Let $(\omega,i)\in\tilde\Omega$. If the zero-th coordinate of $\omega$ is the $j$-th letter of a building block $B_{n_0}$, let the zero-th coordinate of $\varphi(\omega,i)$ be the $(Kj+i)$-th letter of $B'_{n_0}$, otherwise let the zero-th coordinate of $\varphi(\omega,i)$ be a spacer.
    \item For all $j\in\Z$, define the $j$-th coordinate of $\varphi(\omega,i)$ as the zero-th coordinate of $\varphi(\tilde T^j(\omega,i))$.
  \end{itemize}
Then $\varphi$ is an isomorphism between $(\tilde \Omega, \tilde T, \tilde \mu)$ and $(\Omega',T',\mu')$, hence $(\tilde \Omega, \tilde T, \tilde \mu)$ satisfy the assumptions of Theorem~\ref{thm:Sarnak for non WM}, and Sarnak's conjecture holds for $(\tilde \Omega, \tilde T)$. But $\varphi$ is also a topological factor map from $(\tilde\Omega,\tilde T)$ to $(\Omega',T')$,  which proves that Sarnak's conjecture is also valid for $(\Omega',T')$. 
\end{proof}

\end{document}